\documentclass[12pt]{amsart}
\usepackage[margin=1.0in]{geometry}
\usepackage{caption}
\usepackage{graphicx,longtable,wrapfig,xfrac,enumitem,rotating,extpfeil,amsmath,tikz,tikz-cd,amssymb,mathrsfs,float,array,mathtools,bm,amsthm,hyperref,booktabs,makecell,tkz-euclide,glossaries,xcolor}

\DeclareMathOperator{\Id}{Id}

\DeclareMathOperator{\Sq}{Sq}
\DeclareMathOperator{\sq}{sq}
\DeclareMathOperator{\order}{order}
\DeclareMathOperator{\coboundary}{coboundary}

\DeclareMathOperator{\out}{out}
\DeclareMathOperator{\In}{in}

\DeclareMathOperator{\inj}{inj}

\DeclareMathOperator{\rank}{rank}

\DeclareMathOperator{\Tot}{Tot}

\hypersetup{
    colorlinks=true,
    linkcolor=blue,
    filecolor=magenta,
    urlcolor=cyan,
    pdftitle={Overleaf Example}
  }
\newtheorem{theorem}{Theorem}[section]
\newtheorem{proposition}[theorem]{Proposition}
\newtheorem{corollary}[theorem]{Corollary}
\newtheorem{lemma}[theorem]{Lemma}
\theoremstyle{definition}

\newtheorem{definition}[theorem]{Definition}
\newtheorem{convention}[theorem]{Convention}
\newtheorem{notation}[theorem]{Notation}

\newtheorem*{notation*}{Notation}
\theoremstyle{remark}    

\newtheorem{remark}{Remark}[theorem]
\newcommand{\bigslant}[2]{{\raisebox{.2em}{$#1$}\left/\raisebox{-.2em}{$#2$}\right.}}

\makeatletter
\newcommand{\pushright}[1]{\ifmeasuring@#1\else\omit\hfill$\displaystyle#1$\fi\ignorespaces}
\makeatother
\DeclareMathOperator{\Ima}{Im}
\graphicspath{{images/inkscapeimages}}

\newcounter{line}

\newcommand\xline[1]{%
\stepcounter{line}%
\mathrel{\begin{tikzpicture}[baseline= {( $ (current bounding box.south) + (0,-0.5ex) $ )}]
\node[inner sep=.5ex] (\theline) {$\scriptstyle #1$};
\path[draw,-,decorate,
  decoration={segment length=1.2mm,pre=lineto,pre length=4pt}] 
    (\theline.south east) -- (\theline.south west);
\end{tikzpicture}}%
}

\newcommand{\hgline}[3][solid]{
\pgfmathsetmacro{\thetaone}{#2}
\pgfmathsetmacro{\thetatwo}{#3}
\pgfmathsetmacro{\theta}{(\thetaone+\thetatwo)/2}
\pgfmathsetmacro{\phi}{abs(\thetaone-\thetatwo)/2}
\pgfmathsetmacro{\close}{less(abs(\phi-90),0.0001)}
\ifdim \close pt = 1pt
    \draw[blue, #1] (\theta+90:1) -- (\theta-90:1);
\else
    \pgfmathsetmacro{\R}{tan(\phi)}
    \pgfmathsetmacro{\distance}{sqrt(1+\R^2)}
    \draw[blue, #1] (\theta:\distance) circle (\R);
\fi
}

\newcommand{\hpline}[3][]{
  \draw (#3,0) arc(0:180:#3/2-#2/2) [blue] node [pos=0.5,above] {#1}  -- cycle;
}

\newcommand{\hplinesty}[3][solid]{
  \draw (#3,0) arc(0:180:#3/2-#2/2) [blue, #1] -- cycle;
}

\makeglossaries

\newglossaryentry{s p q}{
  name= $s{=} p\circ q$,
  description={$(p,q) \in \partial_W^{n-|V|}\circ \partial_V^n$ is identified with $s$ under $X(\partial_W^{n-|V|}\circ \partial_V^n)$}
}

\newglossaryentry{U}
{
        name=$\partial_U(S{,} T)$,
        description={The set of span elements in $\partial^n_U$ with source in $S$ and target in $T$}
      }
      \newglossaryentry{m}
{
        name=$m_{ab}(z{,}\alpha)$,
        description={$\partial_{ab}(z,\alpha)$}
      }
            \newglossaryentry{source}
{
        name=$s_{\In}$,
        description={The source of $s$}
      }
                  \newglossaryentry{target}
{
        name=$s_{\out}$,
        description={The target of $s$}
      }
      \newglossaryentry{mod coboundary}
{
        name=$\mod\coboundary$,
        description={plus a function $f(z)$, which is a coboundary}
      }
      \newglossaryentry{ab}
{
        name=$a_b$,
        description={$a$ if $b>a$, $a-1$ if $b<a$}
      }
      \newglossaryentry{1P}
{
        name=$1|_P$,
        description={$1$ if $P$ is true. $0$ otherwise}
      }
      \newglossaryentry{mab}
{
        name=$m_{ab}$,
        description={$m_{ab}(z,\alpha)$, if $z$ and $\alpha$ are apparent by context}
      }
      \newglossaryentry{partial c}
{
        name=$\partial_c$,
        description={$\partial_c(z,X_{n+1})$}
}
 \newglossaryentry{partial ab}
{
        name=$\partial_{ab}$,
        description={$\partial_c(z,\alpha)$}
}

\title{A Steenrod square on Khovanov homology and a cup-i product}
\author{Advika Rajapakse}
\address{Department of Mathematics\\University of California\\Los Angeles, CA 90095}
\email{advika@math.ucla.edu}
\thanks{AR was supported by NSF Grant DMS-2136090}

\begin{document}
\maketitle
\begin{abstract}
  Lipshitz-Sarkar defined a stable homotopy type refining Khovanov homology, producing cohomology operations $\Sq^i$ on the Khovanov homology $Kh(L)$ of a link $L$. Later, Mor\'an proposed a sequence of cup-i products on the $\mathbb{F}_2$-coefficient cochain complex of any augmented semi-simplicial object in the Burnside category. Applied to the Khovanov functor, he obtained another sequence of operations $\mathfrak{sq}^n$ on $Kh(L)$, where $\mathfrak{sq}^0$, $\mathfrak{sq}^1$ agree with the usual Steenrod squares. We prove that $\Sq^2$, the first Steenrod operation that cannot be computed from merely homological data, agrees with Mor\'an's $\mathfrak{sq}^2$.
\end{abstract}
\tableofcontents
\section{Introduction}

\subsection{Khovanov homology and Khovanov stable homotopy}
In \cite{MR1740682}, Khovanov categorified the Jones polynomial. That is, for each oriented link diagram $L$, he assigned a bigraded abelian group $Kh^{i,j}(L)$ whose graded Euler characteristic is the (unnormalized) Jones polynomial:
\[
  \chi(Kh^{i,j}(L)) = \sum_{i,j}(-1)^iq^j\rank Kh^{i,j}(L) = (q+q^{-1})V(L).
\]
Furthermore, $Kh^{i,j}(L)$, which is called the \textit{Khovanov homology} of $L$, is an invariant of the underlying link. The discovery of Khovanov homology has led to many striking applications in low-dimensional topology (see \cite{MR2186113,MR2250492,MR2729272,MR4076631}).\par
In \cite{MR3230817}, Lipshitz-Sarkar gave a space-level refinement of Khovanov homology, constructing a space $\mathcal{X}(L)$ whose stable homotopy type is a link invariant, and whose reduced cohomology is $Kh^{i,j}(L)$. This spatial lift has also lead to many topolological applications (see \cite{MR4777698,MR4772951,MR4889247}). For example, this spatial lift allows us to define cohomology operations on Khovanov homology, coming from Steenrod squares $\Sq^i$. The first couple are determined by the identities $\Sq^0 = \Id$, and $\Sq^1 = \beta$, the Bockstein. The next operation, $\Sq^2$, is trickier to compute, since $\Sq^2$ cannot simply be determined by chain-complex-level-data. Nonetheless, Lipshitz-Sarkar \cite{MR3252965} gave a computable formula for $\Sq^2$, giving rise to new computable concordance invariants \cite{MR3189434}.
\subsection{Higher Steenrod squares}
In \cite{MR4153651}, Lawson-Lipshitz-Sarkar give several reformulations of the Khovanov spectrum $\mathcal{X}(L)$, in terms of a (strictly unitary, lax) $2$-functor $F_{Kh}(L):\underline{2}^n\to\mathcal{B}$ from the cube category to the Burnside category. In \cite{MR3611723}, they ask if, just like for $\Sq^1$ and $\Sq^2$, there is a way to compute higher Steenrod squares using purely the data of $F_{Kh}(L)$.\par
In a potential answer to this question, Mor\'an \cite{MoranHigherSquares} defined a sequence of cohomology operations
\[
  \mathfrak{sq}^n: Kh^{i,j}(L;\mathbb{F}_2)\to Kh^{i+n,j}(L;\mathbb{F}_2), \qquad n\geq 0.
\]
These operations are link invariants that indeed only depend on $F$. To define $\mathfrak{sq}^n$, Mor\'an associates to a $2$-functor $F:\underline{2}^N\to\mathcal{B}$, an \textit{augmented semi-simplicial object} $X_\bullet = \Lambda(F)$ in the Burnside category, which has a cochain complex $C^*(X_\bullet)$ satisfying
\[
  \Sigma C^*(X_\bullet,\mathbb{F}_2)\cong C^*(\Tot F;\mathbb{F}_2).
\]
The operations $\mathfrak{sq}^n$ are defined on the homology group $H^*(X_\bullet;\mathbb{F}_2)$, which, in the case $F = F_{Kh}(L;\mathbb{F}_2)$, is just (a grading shift of) $Kh(L)$.\par
So far, it was known that $\mathfrak{sq}^0$, $\mathfrak{sq}^1$ agree with $\Id$ and $\beta$, the identity and Bockstein operations respectively. \par
The following theorem and corollary confirm that $\mathfrak{sq}^2$ does indeed arise as the Steenrod square $\Sq^2$ on the spectrum $\mathcal{X}(L)$. Therefore, we add evidence to the conjecture that $\mathfrak{sq}^n$ indeed do concide with $\Sq^n$ for all $n$.
\begin{theorem}\label{main theorem}
  Let $X_\bullet = \Lambda(F)$. With the canonical identification $\Sigma C^*(X_\bullet;\mathbb{F}_2)\cong C^*(\Tot F;\mathbb{F}_2)$, the operations
  \[
    \mathfrak{sq}^2:H^{*}(X_\bullet;\mathbb{F}_2)\to H^{*+2}(X_\bullet;\mathbb{F}_2),\qquad \Sq^2:H^*(\Tot F;\mathbb{F}_2)\to H^{*+2}(\Tot F;\mathbb{F}_2)
  \]
  agree.
\end{theorem}
\begin{corollary}\label{main corollary}
  Given an oriented link diagram $L$, and let $X_\bullet = \Lambda(F_{Kh}(L))$ denote the augmented semi-simplicial object in the Burnside category, defined by \cite{MoranHigherSquares}. Identify
  \[
    \Sigma^{-n_-+1}C^*(X_\bullet;\mathbb{F}_2) \cong \widetilde{C}^*(\mathcal{X}(L);\mathbb{F}_2)\cong \Sigma^{-n_-}C^*(\Tot F;\mathbb{F}_2)  \cong Kc(L;\mathbb{F}_2).
  \]
  The operation $\mathfrak{sq}^2:Kh^{i,j}(L;\mathbb{F}_2)\to Kh^{i,j}(L;\mathbb{F}_2)$ agrees with the second Steenrod square on the Khovanov stable homotopy type $\mathcal{X}(L)$.
\end{corollary}
Mor\'an defines the maps $\mathfrak{sq}^i$ using a combinatorially defined family of operations
\[
  \smile_i: C^p(X_\bullet;\mathbb{F}_2)\otimes C^q(X_\bullet;\mathbb{F}_2)\to C^{p+q-i}(X_\bullet;\mathbb{F}_2),\qquad i\in \mathbb{Z},
\]
analogous to the cup-i products defined in \cite{MR22071,MR4473678}. Continuing the analogy, he defines $\mathfrak{sq}^i$ by $\mathfrak{sq}^i([\alpha]) = [\alpha\smile_{n-i}\alpha]$ for $\alpha$ any $n$-cocycle. With this construction in mind, we view $\mathfrak{sq}^2$, for the rest of this paper, as an operation from cocycles to cocycles, stating, by abuse of notation, $\mathfrak{sq}^2(\alpha) = \alpha\smile_{n-i}\alpha$. \par
On the other hand, Lipshitz-Sarkar's \cite{MR3230817} derivation of $\sq^2(\alpha)$ comes from a more gometric strategy, using Brown representability to compute $\sq^2(\alpha)$ as a pullback of a cocycle in the Eilenberg-Maclane space $K(\mathbb{Z}/2,n+2)$.
\subsection{Outline of the proof}
Section \ref{background} will cover necessary background and definitions of both formulas $\sq^2$, $\mathfrak{sq}^2$ (which both depend on a few choices). In Section \ref{simplifying} we will simplify Moran's $\mathfrak{sq}^2$ formula as much as we can. Section \ref{difference} will then simplify the difference of the two formulas $\mathfrak{sq}^2(\alpha)-\sq^2(\alpha)$, eventually proving it is a coboundary. We store many of our preliminary lemmas in Appendix \ref{appendix} and \ref{correspondence ordering}, in order to highlight the main argument in Sections \ref{simplifying} and \ref{difference}.
\section{Background and definitions}\label{background}
\subsection{The cube and the Burnside category}
In this section, we will give the necessary setup for the definitions of Steenrod squares in both Mor\'an's work \cite{MoranHigherSquares}, and the work of \cite{MR3252965,rajapakse2025}. The original construction of Khovanov homology \cite{MR1740682} (see also \cite{MR1917056}), starts with an $N$-dimensional cube, where each vertex corresponds to a resolution of an $N$-crossing link diagram $L$. Over each vertex lie generators of the Khovanov chain complex $Kc(L)$, corresponding to labelings of circles, and a differential component on $Kc(L)$ is determined by an arrow traversing down an edge between two generators. Lawson-Lipsihtz-Sarkar \cite{MR3611723, MR4153651} package this data into a (strictly unitary, lax) $2$-functor $F_{Kh}(L): \underline{2}^N\to \mathcal{B}$ from the cube category to the Burnside category. Applying the totalization functor and shifting the gradings down by $N_-$, they again obtain the Khovanov chain complex $Kc(L)$. The main advantage of $F_{Kh}(L)$ is that it defines a symmetric spectrum $|F_{Kh}(L)|$ for which the stable homotopy type of $\Sigma^{-n_-}|F_{Kh}(L)|$ is a link invariant.\par
We define $\underline{2}^N$ to have objects the subsets $A$ of $\{1,\ldots, N\}$, and morphisms the inclusions of subsets. For a definition of the Burnside category $\mathcal{B}$ and a $2$-functor $F$, we refer the reader to \cite{MR4153651,MoranHigherSquares}.
  \subsection{Augmented semi-simplicial sets}
  In \cite{MoranHigherSquares}, Mor\'an packages the data of a 2-functor $F: \underline{2}^N\to \mathcal{B}$ into an \textit{augmented semi-simplicial object} in the Burnside category $\mathcal{B}$.
  \begin{definition}
    Let $\Delta_{\inj*}$ denote the category of possibly empty finite ordinals $[n]$, with only injective order-preserving maps as morphisms. An \textit{augmented semi-simplicial object} $X$ in a (weak) $2$-category $\mathscr{C}$ is a (strictly unitary, lax) $2$-functor from $\Delta_{\inj*}$ to $\mathscr{C}$.
  \end{definition}
  As is customary, we write $X = X_\bullet$, $X_n = X([n])$, and $\partial_U^n = X(\delta_U^n)$ for $U\subseteq [n]$. In practice, $\mathscr{C}$ will be the Burnside category $\mathcal{B}$, so face maps $\delta^n_i:[n-1]\to [n]$ correspond with spans $\partial^n_i:X_n\to X_{n-1}$. We state some notation that will help us define functors in $\mathcal{B}^{\Delta_{\inj*}}$.
  \begin{definition}
    Let $A = (a_0,\ldots, a_{n})$ be a subset of $\{1,\ldots, N\}$, and let $U \subseteq [n]$ be ordered as $(u_1,\ldots,u_k)$. We define $A(U) = (a_{u_1},\ldots, a_{u_k})$.
  \end{definition}
  We can view a $2$-functor $F: (\underline{2}^N)^{\text{op}}\to \mathcal{B}$ as an augmented semi-simplicial object. Indeed, consider the map
  \[
    \Lambda: \mathcal{B}^{(\underline{2}^N)^{\text{op}}}\to \mathcal{B}^{\Delta_{\inj*}}
  \]
  (which is, in fact, functorial), that takes $F$ to the left Kan extension of $F$ along the functor $(\underline{2}^N)^{\text{op}}\to \Delta_{\inj*}$, $A\mapsto [|A|-1]$. In other words, $\Lambda$ takes $F$ to an augmented semi-simplicial object $X_\bullet = \Lambda(F)$ defined by
  \begin{align*}
    X_n &= \coprod_{|A| = n+1} F(A), & n\geq -1\\
    \partial^n_U &= \coprod_{|A| = n+1} F(A\backslash A(U)\subseteq A), & U\subseteq [n]\\
    X(\partial^{n-|V|}_W,\partial^n_V) & = \coprod_{|A| = n+1} F(C\subset B,B\subset A) & V \subset [n], W \subset [n-|V|] 
  \end{align*}
  where $B = A\backslash A(V)$, $C = B\backslash B(V)$.
  \begin{notation}
    Let $(p,q)\in \partial_W^{n-|V|} \times_{X_{n-|V|}}\partial_V^n$. If $X(\partial_W^{n-|V|},\partial_V^n)$ takes $(p,q)$ to $s$, then we say, by abuse of notation, \gls{s p q}.
  \end{notation}
  
  \subsection{Mor\'an's cup-i product}\label{cup i}
  Let $C_*$ denote the chain complex $C_*(X_\bullet;\mathbb{F}_2)$, defined in \cite{MoranHigherSquares}. Mor\'an defines $\mathfrak{sq}^i([\alpha]) = [\alpha\smile_{n-i}\alpha]$, where $\smile_i: C^*\otimes C^*\to C^*$ shifts grading down by $i$ and is defined as the dual to a homomorphism
  \[
    \nabla_i: C_*\to C_*\otimes C_*.
  \] 
  Mor\'an's definition of $\nabla_i$ is analogous to definitions in \cite{MR22071,MR4473678}, where they define a similar operation on the level of simplicial sets. Let $\mathcal{P}_q(n)$ denote the set of $q$-tuples of integers $0\leq a_1\leq a_2\leq\ldots\leq a_q\leq n$ where every number appears at most twice (think of a staircase with unevenly spaced steps), and let $U^+,U^-\subseteq U$ be the subsets of $U$ of positive index elements (see \cite{MoranHigherSquares} for definitions). Mor\'an defines a coproduct of spans
  \begin{equation}\label{span union}
    \nabla^{\binom{n}{q}}=\coprod_{U\in \mathcal{P}_q(n)}\partial_{U^-}\wedge\partial_{U^+},
  \end{equation}
  where $\partial_{U^-}\wedge\partial_{U^+}$ are spans of the form
  \begin{align*}
    X_n\times X_n\longleftarrow \partial_{U^-}\wedge \partial_{U^+}\longrightarrow X_{n-|U^-|}\times X_{n-|U^+|}.
  \end{align*}
  Finally, he defines
  \[
    \nabla_i := \sum_n \mathcal{A}_{\mathbb{F}_2}\left(\nabla^{\binom{n}{n-i}}\circ \Delta_n \right),
  \]
  where $\Delta_n: X_n\to X_n\times X_n$ is the diagonal span and $\mathcal{A}_{\mathbb{F}_2}$ denotes the $\mathbb{F}_2$-linearization functor (see \cite{MoranHigherSquares}). Thus, for $\alpha$ an $n$-cocycle and $z\in X_{n+2}$ viewed as a cochain, we unwind definitions to compute
  \begin{align*}
    \langle\mathfrak{sq}^2(\alpha),z\rangle &= \langle\alpha\smile_{n-2},z\rangle = \left\langle \alpha\otimes \alpha,\mathcal{A}_{\mathbb{F}_2}\bigl(\nabla^{\binom{n+2}{4}}\circ\Delta_{n+2}\bigr)(z)\right\rangle\\
    &= \left\langle\alpha\otimes \alpha,\mathcal{A}_{\mathbb{F}_2}\bigl(\nabla^{\binom{n+2}{4}}\bigr)(z\otimes z)\right\rangle,
  \end{align*}
Finally, we observe that $\alpha\otimes\alpha\in C^n(X_\bullet;\mathbb{F}_2)\otimes C^n(X_\bullet;\mathbb{F}_2)$, which implies that we can replace $\nabla^{\binom{n+2}{4}}$ with the smaller span
  \[
    \nabla^{\binom{n+2}{4}}[2,2] = \coprod_{\substack{U\in \mathcal{P}_4(n+2)\\|U^-| = |U^+|=2}}\partial_{U^-}\wedge\partial_{U^+},
  \]
  and write
  \[
    \langle\mathfrak{sq}^2(\alpha),z\rangle  = \big\langle\alpha\otimes \alpha,\sum_{\substack{U\in \mathcal{P}_4(n+2)\\|U^-| = |U^+|=2}}\mathcal{A}_{\mathbb{F}_2}( \partial_{U^-}\wedge\partial_{U^+})(z\otimes z)\big\rangle.
  \]
\subsection{Background for the $\text{Sq}^2$ formula}
We use a convenient object to define $\sq^2(\alpha)$ called a \textit{cubical special graph structure $\Gamma(z,\alpha)$}, used in \cite{rajapakse2025}. These objects are based of the construction of a \textit{special graph structures}, first defined in \cite{MR4165986} to give a general Steenrod square formula for framed flow categories. We recall a combinatorial formula of the second Steenrod square defined in \cite{rajapakse2025}, which is equivalent to the original combinatorial formula \cite{MR3252965}. But first, we introduce some notation
\begin{definition}
  Given a span
  \begin{align*}
    X_m \longleftarrow \partial^m_U\longrightarrow X_{l},\qquad (l = m - |U|)
  \end{align*}
  and subsets $S\subset X_m$, $T\subset X_l$, we define
  \[
    S \longleftarrow\text{\gls{U}}\longrightarrow T
  \]
  to be the span consisting of elements $s\in \partial^n_U$ with source in $S$, and target in $T$. In the case where $S$ is a singleton set $\{z\}$, we write $\partial_U(\{z\},T)$ as $\partial_U(z,T)$. In practice, we often view a cocycle $\alpha\in C^{l}(X_\bullet;\mathbb{F}_2)$ as a subset $\alpha\subset X_l$, and we use the term $\partial_U(z,\alpha)$. Finally, we define $m_U(S,T)= \#\partial_U(S,T)$. Writing $\{a,b\} = ab$, we have \gls{m}$= \#(\partial_{ab}(z,\alpha))$.
\end{definition}

\begin{theorem}[\cite{MR3252965},\cite{rajapakse2025}]\label{second square identity}
Let $\alpha\in C^*(\Tot F;\mathbb{F}_2)$ be a cocycle. Any facewise boundary matching $\mathfrak{m}$ for $\alpha$ determines a unique cocycle $\sq^2(\alpha)= \sq^2_{\mathfrak{m}}(\alpha)$, defined by
  \begin{align*}
    \langle\sq^2(\alpha),z\rangle &:= \sum_{a<b}ab\cdot m_{ab}(z,\alpha) + \bigslant{\Bigl(\sum_{a<b}(a+b)m_{ab}(z,\alpha)\Bigr)}{2}\\
    &\quad+ \sum_{K\subset \Gamma(z,\alpha)} \bigl(1 + \#\{a\to \overrightarrow{b}\to c\ \vert\  a>b\} + \#\{a\to \overleftarrow{b}\to c\ \vert\  a<b\}\bigr).
  \end{align*}
  Given the canonical identification of $\widetilde{H}^*(|F|;\mathbb{F}_2)$ with $H^*(\Tot F;\mathbb{F}_2)$, we have
  \[
    \Sq^2([\alpha]) = [\sq^2(\alpha)].
  \]
\end{theorem}
A facewise boundary matching is the following data: For $y\in X_{n+1}$, we group $\partial(y,\alpha)$, into ordered pairs of the form $(s,t)$ such that if $(s,t)\in \partial_a\times \partial_b$ is an ordered pair, then $a\leq b$. We fix a facewise boundary matching $\mathfrak{m}$ for $\alpha$ using the following convention:
\begin{convention}
For each $y$, we view $\partial(y,\alpha)$ as a subset of $\partial^{n+1}$ to obtain an induced ordering $s_1<s_2<\ldots s_r$. We define $\mathfrak{m}$ to have the ordered pairs $(s_1,s_2),(s_2,s_3),\ldots, (s_{r-1},s_r)$.
\end{convention}
\subsection{Fixing an  ordering and a boundary matching}
The difficulty in proving the identity $\mathfrak{sq}^2 = \Sq^2$ is that both formulas for $\mathfrak{sq}^2$ and $\Sq^2$ are defined on cocycles $\alpha$ by computing a representative ($\alpha\smile_{n-2}\alpha$ and $\sq^2(\alpha)$ respectively), which creates some ambiguity. On one hand, $\alpha\smile_{n-2}\alpha$ depends on an ordering of all spans $\partial_U^{n}$, and on the other hand, $\sq^2(\alpha)$ depends on a facewise boundary matching for $\alpha$. We fix once and for all an ordering of spans $\partial_U^{n}$ \textit{and} a facewise boundary matching for $\alpha$.
 \begin{definition}\label{ordering convention}
  For the rest of this paper, we fix orderings on our spans $\partial^n_U$. We start by ordering spans of the form $\partial^n_a$ arbitrarily. Now we order spans of the form $\partial^n_{ab}$. For a span $\partial^n_{ab}$, we fix an ordering once and for all called the \textit{left-break} ordering. The ordering is defined as follows:
  \begin{enumerate}
  \item Given $a<b$ and two span elements $s,s'\in \partial^{n+2}_{ab}$, we write $s = p\circ q$, where $q\in \partial_a^{n+2}$, $p\in \partial_{b-1}^{n+1}$, and we write $s' = p'\circ q'$, where $q'\in \partial_a^{n+2}$, $p'\in \partial_{b-1}^{n+1}$.
    \item We use the lexicographic order on  $\partial_a^{n+2}\times \partial_{b-1}^{n+1}$, and define $s<s'$ if and only if $(q,p)<(q',p')$ with respect to this lexicographic ordering.
    \end{enumerate}
    Generalizing the ordering on each $\partial^{n}_a$, we adopt an ordering on $\partial^{n} := \bigsqcup_a\partial^{n}_a$ to be the unique ordering that extends the ordering on each $\partial^{n}_a$ such that $\partial_0^{n}<\partial_1^{n}<\ldots$. In other words, we declare $s<t$ if $s\in \partial^{n}_a$, $t\in \partial^{n}_b$, $a<b$.\par
    We define the \textit{right-break} ordering on $\partial^{n+2}_{ab}$ similarly, but we  write $s = p'\circ q'$ for $q'\in \partial_b^{n+2}$, $p'\in \partial_a^{n+1}$ and ordering lexicographically.
    \end{definition}
    
\begin{remark}
  The term \textit{left-break} ordering comes from viewing a span element $s\in \partial^n_{ab}$ as an interval $\mathcal{I}$ in the moduli space $\mathfrak{f}^{-1}\mathcal{M}(w,u)$ (see \cite{MR3230817, MR4153651} for equivalent definitions). Writing $s = p\circ q$ corresponds to finding the ``leftmost'' end of $\mathcal{I}$.
  \end{remark}
\section{Simplifying the expression $\langle\mathfrak{sq}^2(\alpha),z \rangle$}\label{simplifying}
Let $\alpha$ be an $n$-cocycle. In this section, we try to simplify $\langle\mathfrak{sq}^2(\alpha),z \rangle$  as much as possible before comparing it with $\langle\sq^2(\alpha),z\rangle$. Recall from Section \ref{cup i} our formula
\begin{equation*}
    \langle\mathfrak{sq}^2(\alpha),z\rangle  = \big\langle\alpha\otimes \alpha,\sum_{\substack{U\in \mathcal{P}_4(n+2)\\|U^-| = |U^+|=2}}\mathcal{A}_{\mathbb{F}_2}( \partial_{U^-}\wedge\partial_{U^+})(z\otimes z)\big\rangle.
  \end{equation*}
  Mor\'an's construction of $\partial_{U^-}\wedge\partial_{U^+}$ is slightly different based on whether $n$ is even or odd. For ease of reading, we follow the computation in the case that $n$ is even. The computation for $n$ odd follows similarly and results in an almost identical formula. We now expand the interior term
  \begin{equation}\label{interior term}
    \sum_{\substack{U\in \mathcal{P}_4(n+2)\\|U^-| = |U^+|=2}}\mathcal{A}_{\mathbb{F}_2}( \partial_{U^-}\wedge\partial_{U^+}),
  \end{equation}
  and in the following equalities, we drop the ``$\mathcal{A}_{\mathbb{F}_2}$'' term to lighten the notation, leaving it implicit. The term (\ref{interior term}) can be written as
\begingroup
\allowdisplaybreaks
\begin{align*}
  &\sum_{\substack{U\in \mathcal{P}^0_4(n+2)\\|U^+| = |U^-| = 2}}\partial_{U^-}\wedge\partial_{U^+}+ \sum_{\substack{U\in \mathcal{P}^1_4(n+2)\\|U^+| = |U^-| = 2}}\partial_{U^-}\wedge\partial_{U^+} + \sum_{\substack{U\in\mathcal{P}^2_4(n+2)\\|U^+| = |U^-| = 2}}\partial_{U^-}\wedge\partial_{U^+}\\
  &=\Biggl(\sum_{\substack{a<b<c<d\\\text{$a,b$ even}\\\text{$c,d$ odd}}}\partial_{ad}\wedge\partial_{bc} + \sum_{\substack{a<b<c<d\\\text{$c,d$ even}\\\text{$a,b$ odd}}}\partial_{bc}\wedge\partial_{ad} + \sum_{\substack{a<b<c<d\\\text{$a,d$ even}\\\text{$b,c$ odd}}}\partial_{ab}\wedge\partial_{cd} + \sum_{\substack{a<b<c<d\\\text{$b,c$ even}\\\text{$a,d$ odd}}}\partial_{cd}\wedge\partial_{ab}\\
  &\quad + \sum_{\substack{a<b<c<d\\\text{$a,b,c,d$ odd}}}\partial_{bd}\wedge\partial_{ac} + \sum_{\substack{a<b<c<d\\\text{$a,b,c,d$ odd}}}\partial_{ac}\wedge\partial_{bd}\Biggr)\\
  &\quad+ \Biggl(\sum_{\substack{a<c<b\\\text{$a$ even}\\\text{$b$ odd}}}\partial_{ac}\wedge\partial_{cb} + \sum_{\substack{a<c<b\\\text{$a$ odd}\\\text{$b$ even}}}\partial_{cb}\wedge\partial_{ac} + \sum_{\substack{c<a<b\\\text{$a$, $b$ even}}}\partial_{ca}\wedge\partial_{cb} + \sum_{\substack{c<a<b\\\text{$a$, $b$ odd}}}\partial_{ca}\wedge\partial_{cb}\\
  &\quad+ \sum_{\substack{a<b<c\\\text{$a$, $b$ even}}}\partial_{ac}\wedge\partial_{bc} + \sum_{\substack{a<b<c\\\text{$a$, $b$ odd}}}\partial_{bc}\wedge\partial_{ac}\Biggr) + \sum_{a<b} \partial_{ab}\wedge\partial_{ab}\\
  &= I + II + III
\end{align*}
\endgroup
\begin{notation}
  When $z\in X_{n+2}$ is apparent by context, we will abbreviate $\partial_c(z,X_{n+1})$ as \gls{partial c}, $\partial_{ab}(z,\alpha)$ as \gls{partial ab}, and $m_{ab}(z,\alpha)$ by \gls{mab}. Furthermore, when $y\in X_{n+1}$, we write $m_a(y,\alpha)$ as $m_a(y)$.
\end{notation}
We write
\begin{align*}
  \langle\mathfrak{sq}^2(\alpha),z\rangle
  &= \langle\alpha\otimes\alpha, (I + II + III) (z\otimes z)\rangle\\
  &= \langle\alpha\otimes\alpha, I (z\otimes z)\rangle + \langle\alpha\otimes\alpha, II (z\otimes z)\rangle + \langle\alpha\otimes\alpha, III (z\otimes z)\rangle\\
  &= \mathbf{I} + \mathbf{II} + \mathbf{III},
\end{align*}
To compute $\mathbf{I}$, note that
\begin{align*}
  I&=\sum_{\substack{a<b<c<d\\\text{$a,b$ even}\\\text{$c,d$ odd}}}\partial_{ad}\times\partial_{bc} + \sum_{\substack{a<b<c<d\\\text{$c,d$ even}\\\text{$a,b$ odd}}}\partial_{bc}\times\partial_{ad} + \sum_{\substack{a<b<c<d\\\text{$a,d$ even}\\\text{$b,c$ odd}}}\partial_{ab}\times\partial_{cd} + \sum_{\substack{a<b<c<d\\\text{$b,c$ even}\\\text{$a,d$ odd}}}\partial_{cd}\times\partial_{ab}\\
  &\quad + \sum_{\substack{a<b<c<d\\\text{$a,b,c,d$ odd}}}\partial_{bd}\times\partial_{ac} + \sum_{\substack{a<b<c<d\\\text{$a,b,c,d$ odd}}}\partial_{ac}\times\partial_{bd}
\end{align*}
using the fact that $\partial_{ab}\wedge\partial_{cd} = \partial_{ab}\times\partial_{cd}$ if $a,b,c,d$ are all distinct. Therefore.
\begin{align*}
  \mathbf{I} &= \sum_{\substack{a<b<c<d\\a,b \text{ even}\\c,d\text{ odd}}}m_{ad}m_{bc} + \sum_{\substack{a<b<c<d\\c,d \text{ even}\\a,b\text{ odd}}}m_{bc}m_{ad} + \sum_{\substack{a<b<c<d\\a,d \text{ even}\\b,c\text{ odd}}}m_{ab}m_{cd} + \sum_{\substack{a<b<c<d\\b,c \text{ even}\\a,d\text{ odd}}}m_{cd}m_{ab}\\
             &\quad+ \sum_{\substack{a<b<c<d\\a,b,c,d \text{ odd}}}m_{bd}m_{ac}+ \sum_{\substack{a<b<c<d\\a,b,c,d \text{ even}}}m_{ac}m_{bd},
\end{align*}
To compute $\mathbf{II}$, we use the following notation to rewrite $II$.
\begin{notation}
  Given integers $a,b$ such that $a\neq b$, define
  \[
    \text{\gls{ab}} =
    \begin{cases*}
      a & if $a<b$\\
      a-1 & if $a>b$
    \end{cases*}
  \]
  We only use the letters $a,b,c,d$ in this notation throughout the paper. The convenience of this definition is seen in by the isomorphism of spans
  \[
    \partial^{n-1}_{a_b}\circ\partial_b^{n}\cong \partial_{ab}^{n}\cong  \partial^{n-1}_{b_a}\circ\partial_a^{n}.
  \]
\end{notation}
Now we can rewrite $II$ as
\begin{equation*}
  II = \sum_c\Bigl(\sum_{\substack{a,b\neq c\\a<b\\\text{$a_c,b_c$ even,}}} \partial_{ac}\wedge\partial_{bc} + \sum_{\substack{a,b\neq c\\a<b\\\text{$a_c,b_c$ odd,}}} \partial_{bc}\wedge\partial_{ac}\Bigr).
\end{equation*}
We observe from \cite{MoranHigherSquares} that if $a,b,c$ are disjoint, then $\partial_{ac}\wedge\partial_{bc}$, can be viewed as a subset of $\partial_{ac}\times\partial_{bc}$, consisting of exactly the pairs $(s_0,t_0)$ such that if  $s_0= p\circ s\in \partial_{a_c}\circ\partial_c$, $t_0= q\circ t\in \partial_{b_c}\circ\partial_c$, then $s<t$. We write $\mathbf{II}$ after introducing the following notation.
\begin{definition}
  Let $A\longleftarrow E\longrightarrow B$ be a span, and let $s\in E$. We denote \gls{source}$\in A$, \gls{target}$\in B$ as the source and target elements of $s$.
\end{definition}
\begin{align*} 
  \mathbf{II} &=\sum_{\substack{s,t\in \partial_c\\s<t}}\Biggl(\sum_{\substack{a,b\neq c\\a<b\\\text{$a_c,b_c$ even}}}m_{a_c}(s_{\out})m_{b_c}(t_{\out}) +\sum_{\substack{a,b\neq c\\a>b\\\text{$a,b$ odd}}}m_{a_c}(s_{\out})m_{b_c}(t_{\out})\Biggr)\\
  &=\sum_{\substack{s,t\in \partial_c\\s<t}}\Biggl(\sum_{\substack{a<b\\\text{$a,b$ even}}}m_a(s_{\out})m_b(t_{\out}) +\sum_{\substack{a>b\\\text{$a,b$ odd}}}m_a(s_{\out})m_b(t_{\out})\Biggr).
\end{align*}
Referring directly to the definition of $\partial_{ab}\wedge\partial_{ab}$ in \cite{MoranHigherSquares}, we find
\begin{equation}\label{III def}
  \mathbf{III} := \sum_{a<b}\sum_{s,t\in \partial_{ab}}\#\left\{\text{positive maximal chains } (W^\|,W^\circ) \prec \left(\emptyset, \{a,b\}\right) \text{ of $(s,t)$-good pairs}\right\}.
\end{equation}
We rewrite $\mathbf{III}$ more explicitly, but we first introduce some more notation.
  \begin{definition}
    Let $P$ be a statement. We define
    \[
      \text{\gls{1P}} =
      \begin{cases*}
        1 &if $P$ is true\\
        0&if $P$ is false.
      \end{cases*}
    \]    
  \end{definition}
\begin{lemma}\label{III rewritten}
  We have
  \begin{equation*}                
  \mathbf{III} = \sum_{a<b} \#\left\{\{s,t\}\subset\partial_{ab}\Bigm|\ \text{left-break order of $\{s,t\}$ disagrees with right-break order}\right\}.
  \end{equation*}
\end{lemma}
\begin{proof}
  In the formula (\ref{III def}) for $\mathbf{III}$, only the indices where $s<t$ contribute, since there are no positive maximal chains if $s>t$ (see \cite{MoranHigherSquares}). Therefore, it suffices to prove that for any $s,t$ such that $s<t$, the number of positive maximal chains of $(s,t)$-good pairs is $1\pmod 2$ if and only if the left-break order of $\{s,t\}$ disagrees with the right-break order. By our ordering convention, the latter statement is equivalent to the the right-break order being $(t,s)$. Our lemma is thus a consequence of the following statement:
  \[
    \#\{\text{positive maximal chains of $(s,t)$-good pairs}\} \equiv
    \begin{cases*}
      1 & if right-break order is $(t,s)$\\
      0 & if right-break order is $(s,t)$
    \end{cases*}\pmod 2.
  \]
  Write $s = p_s\circ q_s\in \partial_{b-1}\circ \partial_a$ and $s = p'_s\circ q'_s\in \partial_{a}\circ \partial_b$. Similarly, we write $t = p_t\circ q_t\in \partial_{b-1}\circ \partial_a$ and $t = p'_t\circ q'_t\in \partial_{a}\circ \partial_b$ (see Figure \ref{compositions} for a diagram). 
  \begin{figure}
    \def\T{0.5cm} \def\P{0.3cm}
    \begin{tabular}{ m{6.0cm} m{6.0cm}}
    \begin{tikzpicture}[scale=1.6]
  \path[draw, <-, shorten <=\T,shorten >=\T, thick] (0,1) -- (1,2) node [midway, label={[label distance=-0.2cm]135:$\partial_a^{n+2}$}] {};
  \path[draw, <-, shorten <=\T,shorten >=\T, thick] (2,1) -- (1,2) node [midway, label={[label distance=-0.2cm]45:$\partial_b^{n+2}$}] {};
  \path[draw, <-, shorten <=\T,shorten >=\T, thick] (1,0) -- (2,1) node [midway, label={[label distance=-0.25cm]-45:$\partial_a^{n+1}$}] {};
  \path[draw, <-, shorten <=\T,shorten >=\T, thick] (1,0) -- (0,1) node [midway, label={[label distance=-0.35cm]225:$\partial_{b-1}^{n+1}$}] {};
  \path[draw, <-, shorten <=\T,shorten >=\T] (1,0) to [bend left = 15] node[pos=0.5,left] {$s$} (1,2);
  \path[draw, <-, shorten <=\T,shorten >=\T, red] (1,0) to [bend right = 15] node[pos=0.5,right] {$t$} (1,2);

   \node at (1, 2) {$X_{n+2}$};
   \node at (0, 1) {$X_{n+1}$};
   \node at (2, 1) {$X_{n+1}$};
   \node at (1, 0) {$X_n$};
 \end{tikzpicture} &
 \begin{tikzpicture}[scale=1.6]
   \path[draw, <-, shorten <=\T,shorten >=\T] (0,1) to [bend left = 20] node [pos=0.5,above left] {$q_s$} (1,2);
   \path[draw, <-, shorten <=\T,shorten >=\T, red] (0,1) to [bend right = 20] node [pos=0.5,below right] {$q_t$} (1,2);
   \path[draw, <-, shorten <=\T,shorten >=\T, red] (2,1) to [bend left = 20] node [pos=0.5,below left] {$q'_t$} (1,2);
   \path[draw, <-, shorten <=\T,shorten >=\T] (2,1) to [bend right = 20] node [pos=0.5,above right] {$q'_s$} (1,2);
   \path[draw, <-, shorten <=\T,shorten >=\T, red] (1,0) to [bend left = 20] node [pos=0.5,above left] {$p'_t$} (2,1);
   \path[draw, <-, shorten <=\T,shorten >=\T] (1,0) to [bend right = 20] node [pos=0.5,below right] {$p'_s$} (2,1);
   \path[draw, <-, shorten <=\T,shorten >=\T] (1,0) to [bend left = 20] node [pos=0.5,below left] {$p_s$} (0,1);
   \path[draw, <-, shorten <=\T,shorten >=\T, red] (1,0) to [bend right = 20] node [pos=0.5,above right] {$p_t$} (0,1);

   \node at (1, 2) {$X_{n+2}$};
   \node at (0, 1) {$X_{n+1}$};
   \node at (2, 1) {$X_{n+1}$};
   \node at (1, 0) {$X_n$};
 \end{tikzpicture}
      \end{tabular}
      \caption{Left: the elements $s,t$ in the span $\partial_{ab} = \partial_{ab}(z)$ The thick arrows denote the spans $\partial^{n+2}_{a},\partial^{n+2}_{b},\partial^{n+1}_{b-1},\partial^{n+1}_{a}$, . Right: We write both $s$ and $t$ as compositions in  $\partial_{b-1}\circ\partial_a$ (composing along the left side of the diagram) or in $\partial_{a}\circ\partial_b$ (composing along the right side of the diagram).}
      \label{compositions}
  \end{figure}
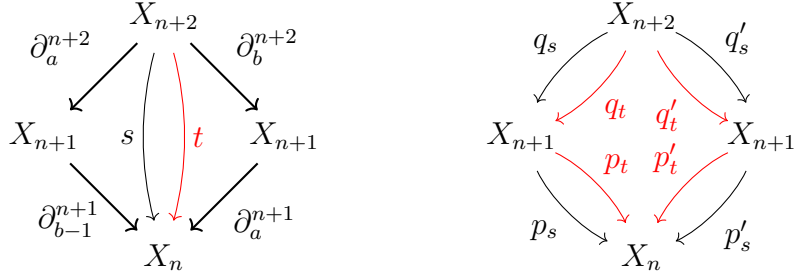
  Now we measure the number of positive maximal chains of $(s,t)$-good pairs. This measurement depends on whether $q_s = q_t$, and $q'_s = q'_t$. There are $4$ total cases, but we will only consider two, as the rest are a similar verification. We first introduce notation that will help our measurement.
  First consider the case where none of the equalities hold. Then the only possible positive maximal chains are $(\emptyset, \{a\})\prec (\emptyset, \{a,b\})$, and $(\emptyset, \{b\})\prec (\emptyset, \{a,b\})$. The first chain is positive if and only if (we are using the fact that $n$ is even here), $q_s<q_t$. The second chain is positive if and only if $q'_s<q'_t$. But note $s<t$ implies $q_s<q_t$ by Definition \ref{ordering convention}. Therefore, the number of positive maximal chains of $(s,t)$-good pairs is
  \[
    1 + 1|_{q'_s<q'_t} = 1_{q'_t<q'_s} = \begin{cases*} 1 & if right-break order of $\{s,t\}$ is $(t,s)$\\
      0 &  otherwise.
      \end{cases*}\pmod 2.
  \]\par
  Now suppose that $q_t = q_s$, but $q'_t\neq q'_s$. The only possible positive maximal chains are $(\emptyset, \{b\})\prec (\emptyset, \{a,b\})$, and $(\{a\},\{b\})\prec (\emptyset, \{a,b\})$. The first chain is positive if and only if $q'_s<q'_t$ and the second chain is positive if and only if $p_s<p_t$. But note that $s<t$ implies $p_s<p_t$ by Definition \ref{ordering convention}. We see once again that the number of positive maximal chains is $1\pmod 2$ is and only if the right-break order is $(t,s)$.
\end{proof}
We now turn back to expanding the term $\mathbf{I}$, replacing the elaborate conditions in the sum indexing with terms that only equal $1$ at the appropriate indices.
  \begin{align*}
  &\mathbf{I} = \sum_{a<b<c<d}(abd+acd+abc+bcd+ab+cd)m_{ad}m_{bc}\notag\\
  &\quad + \sum_{a<b<c<d}(abd+acd+abc+bcd+bc+ad)m_{ab}m_{cd}\notag\\
  &\quad +\sum_{a<b<c<d}(abd+acd+abc+bcd+ab+ac+ad+bc+bd+cd+a+b+c+d+1)m_{ac}m_{bd},
  \end{align*}
  which we rewrite as
  \begin{align*}
    &\sum_{a<b<c<d}(abd+acd+abc+bcd)(m_{bc}m_{ad} + m_{v}m_{ab} + m_{ac}m_{bd})\\
  &\quad+\sum_{a<b<c<d}\Bigl((ab+cd)(m_{ad}m_{bc} + m_{ac}m_{bd})+(ad+bc)(m_{ab}m_{cd} + m_{ac}m_{bd}) + (ac+bd)(m_{ac}m_{bd})\Bigr)\\
    &\quad+ \sum_{a<b<c<d}(a+b+c+d)m_{ac}m_{bd} + \sum_{a<b<c<d}m_{ac}m_{bd}
  \end{align*}

\begin{lemma}\label{3term 0}
  The first summand in our formula for $\mathbf{I}$ is equal to $0$. In other words,
  \begin{equation}\label{3term}
    \sum_{a<b<c<d}(abd+acd+abc+bcd)(m_{bc}m_{ad} + m_{cd}m_{ab} + m_{ac}m_{bd})=0.
    \end{equation}
\end{lemma}
\begin{proof}
  The left-hand side of Equation (\ref{3term}) can be written as
  \begin{align*}
    &\sum_{c<d}cd\cdot m_{cd}\underbrace{\sum_{a<b}(a+b)m_{ab}}_{=0} + \sum_c c\sum_{\substack{a<b\\a,b\neq c}}(a+b) m_{ac}m_{bc} + \sum_{a<b}\underbrace{ab(a+b)}_{=0}m_{ab}m_{ab}\\
    & = \sum_c c\sum_{\substack{a<b\\a,b\neq c}}(a+b) m_{ac}m_{bc} = 0 \pmod 2 \qquad(\text{by Lemma \ref{one intersection}}).\qedhere
  \end{align*}
\end{proof}.
\begin{lemma}\label{a+b lemma}
  \begin{equation}\label{a+b identity}
    \sum_{\substack{a<b,\ c<d\\\text{disjoint}}}(a+b)(c+d)m_{ab}m_{cd} = \sum_c\sum_{\substack{a,b\neq c \\ a<b}}ab\cdot m_{ac}m_{bc} + \sum_{a<b} ab\cdot m_{ab}
\end{equation}
\end{lemma}
\begin{proof}
\begin{align*}
  0 &= \Bigl(\sum_{a<b}(a+b)m_{ab}(z)\Bigr)^2/2 \\
    &= \sum_{\substack{a<b,\ c<d\\\text{disjoint}}}(a+b)(c+d)m_{ab}m_{cd} + \sum_c\sum_{\substack{a,b\neq c\\a<b}}(a+c)(b+c)m_{ac}m_{bc} \\
    &\quad+ \Bigl(\sum_{a<b}\bigl((a+b)m_{ab}\bigr)^2\Bigr)/2\\
    &= A + B + C
\end{align*}
Now we write $B$ as follows:
\begingroup
\allowdisplaybreaks
\begin{align*}
  B &= \sum_{c}\sum_{\substack{a,b\neq c\\a<b}}(c+c(a+b)+ab)m_{ac}m_{bc}\\
    &= \sum_c c\sum_{\substack{a,b\neq c\\a<b}}m_{ac}m_{bc} +  \sum_c\sum_{\substack{a<b\\\neq c}}c(a+b)m_{ac}m_{bc}+ \sum_c\sum_{\substack{a,b\neq c\\a<b}}ab\cdot m_{ac}m_{bc}\\
    &= \sum_c c\sum_{\substack{a,b\neq c\\a<b}}m_{ac}m_{bc} + 0+ \sum_c\sum_{\substack{a,b\neq c\\a<b}}ab\cdot m_{ac}m_{bc}\qquad \text{(By Lemma \ref{one intersection})}\\
  &= B' + \sum_c\sum_{\substack{\substack{a,b\neq c\\a<b}}}(am_{ac})(bm_{bc})
\end{align*}
\endgroup
where $B'$ is defined by
\begingroup
\allowdisplaybreaks
\begin{align*}
  B' &= \sum_c c\left(\binom{\sum_{a\neq c}m_{ac})}{2}-\sum_{a\neq c}\binom{m_{ac}}{2}\right)\\
      &= \sum_c c\frac{\sum_{a\neq c}m_{ac}}{2}+ \sum_{a<b}(a+b)\binom{m_{ab}}{2}\\
      &= \frac{\sum_cc\sum_{a\neq c}m_{ac}}{2} + \sum_{a<b}(a+b)\binom{m_{ab}}{2}\\
  &= \frac{\sum_{a<b}(a+b)m_{ab}}{2} + \sum_{a<b}(a+b)\binom{m_{ab}}{2}
\end{align*}
\endgroup
Now we can rewrite $C$ as the following:
\begingroup
\allowdisplaybreaks
\begin{align*}
  C &= \Bigl(\sum_{a<b}(a^2+2ab+b^2)m_{ab}^2\Bigr)/2\\
    &= \sum_{a<b} ab\cdot m_{ab}^2 + \Bigl(\sum_{a<b}(a^2+b^2)m_{ab}^2\Bigr)/2\\
  &= \sum_{a<b} ab\cdot m_{ab} + \sum_{a<b} (a+b) \binom{m_{ab}}{2} + \Bigl(\sum_{a<b}(a+b)m_{ab}\Bigr)/2,
\end{align*}
\endgroup
with the third equality being from the equation
\begin{align*}
  \Bigl(\sum_{a<b}(a^2+b^2)m_{ab}^2\Bigr)/2
  &= \sum_{a<b}\Bigl(a m_{ab}(a m_{ab}-1)\Bigr)/2 + \Bigl(b m_{ab}(b m_{ab}-1)\Bigr)/2\\
  &\quad+ \sum_{a<b}\Bigl((a+b)m_{ab}\Bigr)/2\\
  &=\sum_{a<b}\binom{am_{ab}}{2} + \binom{bm_{ab}}{2} + \Bigl(\sum_{a<b}(a+b)m_{ab}\Bigr)/2\\
  &= \sum_{a<b}(a+b)\binom{m_{ab}}{2}+\sum_{a<b}m_{ab}\left(\binom{a}{2}+  \binom{b}{2}\right)\qquad \text{(Using Lemma \ref{Leibniz})}\\
  &\quad + \Bigl(\sum_{a<b}(a+b)m_{ab}\Bigr)/2\\
  &= \sum_{a<b}(a+b)\binom{m_{ab}}{2} + 0 + \Bigl(\sum_{a<b}(a+b)m_{ab}\Bigr)/2\qquad \text{(Using Lemma \ref{fa+fb})}
\end{align*}
Putting $A+B+C$ together, we obtain Equation (\ref{a+b identity}) that is the statement of our lemma.
\end{proof}
We are about to state our final formula for $\langle\mathfrak{sq}(\alpha),z\rangle$, but we wish to quotient out terms which come from a coboundary. We shall use the following definition.
  \begin{definition}
    Let $f, g\in C^{n+2}(X_\bullet;\mathbb{F}_2)$ be two cochains. We say that $f=g$\gls{mod coboundary} if there exists a coboundary $\delta \omega\in C^{n+2}(X_{\bullet};\mathbb{F}_2)$ such that $f -g = \delta \omega$. We also abuse notation and say $f(z)=g(z)\mod\coboundary$. 
  \end{definition}
  
\begin{proposition}\label{Moran sq2 rewrite}
  We can rewrite $\langle\mathfrak{sq}(\alpha),z\rangle$ as
  \begin{equation}\label{simplified I terms}
    \begin{aligned}
      \langle\mathfrak{sq}(\alpha),z\rangle &=  \sum_c\sum_{\substack{a,b\neq c\\a<b}}ab\cdot m_{ac}m_{bc} + \sum_{a<b} ab\cdot m_{ab} + \sum_{a<b<c<d}(a+b+c+d)m_{ac}m_{bd}\\
    &\quad + \sum_{K\subset \Gamma(z,\alpha)} (1 + \#\{a\to \overrightarrow{b}\to c\ \vert\  a>b\} + \#\{a\to \overleftarrow{b}\to c\ \vert\  a<b\})\\
    &\quad + \sum_{\substack{s,t\in \partial_c\\s<t}}\sum_{\substack{c<a<b\\a<b<c\\b<c<a}}m_{a_c}(s_{\out})m_{b_c}(t_{\out}) + \sum_{a<b}\binom{m_{ab}}{2}\\
    &\quad+ \sum_{\substack{s,t\in \partial_c\\s<t}}\sum_{\substack{a<b\\\text{$a,b$ even}}}m_{a}(s_{\out})m_{b}(t_{\out}) + \sum_{\substack{s,t\in \partial_c\\s<t}}\sum_{\substack{a<b\\\text{$a,b$ odd}}}m_b(s_{\out})m_a(t_{\out}).
  \end{aligned}
\end{equation}
mod coboundary.
\end{proposition}
\begin{proof}
  Recall that $\langle\mathfrak{sq}(\alpha),z\rangle = \mathbf{I} + \mathbf{II} + \mathbf{III}$, where
  \begin{equation}\label{I recall}
  \begin{aligned}
    \mathbf{I} &= \sum_{a<b<c<d}(abd+acd+abc+bcd)(m_{bc}m_{ad} + m_{v}m_{ab} + m_{ac}m_{bd})\notag\\
               &+\sum_{a<b<c<d}\Bigl((ab+cd)(m_{ad}m_{bc} + m_{ac}m_{bd})+(ad+bc)(m_{ab}m_{cd} + m_{ac}m_{bd}) + (ac+bd)(m_{ac}m_{bd})\Bigr)\\
    &\quad + \sum_{a<b<c<d}(a+b+c+d)m_{ac}m_{bd} + \sum_{a<b<c<d}m_{ac}m_{bd}
  \end{aligned}
  \end{equation}
  From Lemma (\ref{3term 0}), we see that the first sum equals $0$. By subtracting Equation (\ref{ab+cd=0}) in Lemma \ref{ab+cd lemma} from Equation (\ref{a+b identity}) in Lemma \ref{a+b lemma}, we have that the second sum in Equation (\ref{I recall}) equals
  \[
    \sum_c\sum_{\substack{a,b\neq c\\a<b}}ab\cdot m_{ac}m_{bc} + \sum_{a<b} ab\cdot m_{ab}.
  \]
  Finally, we recall from Proposition \ref{cross counting} that
  \begin{equation}\label{ac*bd equation}
    \begin{aligned}
      &\sum_{a<b<c<d}m_{ac}m_{bd} = \sum_{\substack{s,t\in \partial_c\\s<t}}\sum_{\substack{c<a<b\\a<b<c\\b<c<a}}m_{a_c}(s_{\out})m_{b_c}(t_{\out})\\
      &\quad+ \sum_{a<b} \#\left\{\{s,t\}\subset\partial_{ab}(z,\alpha)\Bigm|\ \text{left-break order of $\{s,t\}$ agrees with right-break order}\right\}\\
      &\quad+ \sum_K (1 + \#\{a\to \overrightarrow{b}\to c\ \vert\  a>b\} + \#\{a\to \overleftarrow{b}\to c\ \vert\  a<b\})\mod \coboundary.
    \end{aligned}
  \end{equation}
  Now recall from Lemma \ref{III rewritten} that
  \[
    \mathbf{III} = \sum_{a<b} \#\left\{\{s,t\}\subset\partial_{ab}(z,\alpha)\Bigm|\ \text{left-break order of $\{s,t\}$ disagrees with right-break order}\right\}
  \]
  We add this term to the second line of Equation (\ref{ac*bd equation}) to obtain $\sum_{a<b}\binom{m_{ab}}{2}$.\par
  We have shown that the sum $\mathbf{I} + \mathbf{III}$ gives us all but the last line in Equation (\ref{simplified I terms}). The last line is identically $\mathbf{II}$, finishing our proof.
\end{proof}
\section{The difference of the two formulas is a coboundary}\label{difference}
Our next task is to prove the difference $\sq^2(\alpha) - \mathfrak{sq}^2(\alpha)$ is a coboundary. Using Proposition \ref{Moran sq2 rewrite}, we can write the difference as the following.
\begingroup
\allowdisplaybreaks
\begin{align*}
  &\langle\sq^2(\alpha) - \mathfrak{sq}^2(\alpha),z\rangle \\
  &= \sum_{\substack{a,b\neq c\\a<b}}ab\cdot m_{ac}m_{bc} + \sum_{\substack{s,t\in \partial_c\\s<t}}\sum_{\substack{c<a<b\\a<b<c\\b<c< a}}m_{a_c}(s_{\out})m_{b_c}(t_{\out})\\
  &\quad + \sum_{\substack{s,t\in \partial_c\\s<t}}\sum_{\substack{\text{$a,b$ even}\\a<b}}m_{a}(s_{\out})m_{b}(t_{\out}) + \sum_{\substack{s,t\in \partial_c\\s<t}}\sum_{\substack{\text{$a,b$ odd}\\a>b}}m_{a}(s_{\out})m_{b}(t_{\out})\\
  &\quad+ \sum_{a<b<c<d}(a+b+c+d)m_{ac}m_{bd} + \sum_{a<b}\binom{m_{ab}}{2} + \Bigl(\sum_{a<b}(a+b)m_{ab}(z)\Bigr)/2\\
\end{align*}
\endgroup
The rest of this section is devoted to progressively simplifying the above formula, eventually proving it is $0$ mod coboundary. The following proposition is our first (and most involved) step:
\begin{proposition}\label{big simplify}
  \begin{align*}
    &\sum_{\substack{a,b\neq c\\a<b}}ab\cdot m_{ac}m_{bc} + \sum_{\substack{s,t\in \partial_c\\s<t}}\sum_{\substack{c<a<b\\a<b<c\\b<c<a}}m_{a_c}(s_{\out})m_{b_c}(t_{\out})\\
    &\quad+ \sum_{\substack{s,t\in \partial_c(z)\\s<t}}\Bigl(\sum_{\substack{\text{$a,b$ even}\\a<b}}m_{a}(s_{\out})m_{b}(t_{\out}) + \sum_{\substack{\text{$a,b$ odd}\\b<a}}m_{a}(s_{\out})m_{b}(t_{\out})\Bigr)\\
    &=  \sum_{\substack{s,t\in \partial_c\\s<t}}\sum_{\substack{c<a<b\\a<b<c\\b<c<a}}(a+b)m_{a_c}(s_{\out})m_{b_c}(t_{\out}) + \sum_{\substack{c,\\s\in \partial_c}}\sum_{\substack{c<a<b\\a<b<c}}(a+b)m_{a_c}(s_{\out})m_{b_c}(s_{\out})\\
    & + \sum_{\substack{s,t\in \partial_c\\s<t}}\Bigl(\sum_{\substack{a,b,c\\\text{disjoint,}\\c<a}}m_{a_c}(s_{\out})m_{b_c}(t_{\out})\Bigr) + \sum_{{\substack{c,\\s\in \partial_c}}}\sum_{c<a<b}m_{a_c}(s_{\out})m_{b_c}(t_{\out}) \mod \coboundary.
  \end{align*}
\end{proposition}
\begin{proof}
  For a given $s<t$, we have
  \begingroup
  \allowdisplaybreaks
  \begin{align*}
    &\sum_{\substack{\text{$a,b$ even}\\a<b}}m_{a}(s_{\out})m_{b}(t_{\out}) + \sum_{\substack{\text{$a,b$ odd}\\b<a}}m_{a}(s_{\out})m_{b}(t_{\out})\\
    &= \sum_{a<b}(a+1)(b+1)m_{a}(s_{\out})m_{b}(t_{\out}) + \sum_{a>b}(ab)m_{a}(s_{\out})m_{b}(t_{\out})\\
    &= \sum_{\substack{a,b\neq c\\a<b}}(a_c+1)(b_c+1)m_{a_c}(s_{\out})m_{b_c}(t_{\out})+ \sum_{\substack{a,b\neq c\\a>b}}(a_cb_c)m_{a_c}(s_{\out})m_{b_c}(t_{\out})\\
    &= \sum_{\substack{a,b\neq c\\a<b}}(a + 1|_{a<c})(b+1|_{b<c})m_{a_c}(s_{\out})m_{b_c}(t_{\out}) + \sum_{\substack{a,b\neq c\\a>b}}(a+1|_{c<a})(b+1|_{c<b})m_{a_c}(s_{\out})m_{b_c}(t_{\out})\\
    &= \sum_{\substack{a,b\neq c\\a<b}}(ab+b|_{a<c}+a|_{b<c}+1|_{b<c})m_{a_c}(s_{\out})m_{b_c}(t_{\out})\\
    &\quad+ \sum_{\substack{a,b\neq c\\a>b}}(ab+b|_{c<a}+a|_{c<b}+1|_{c<b})m_{a_c}(s_{\out})m_{b_c}(t_{\out}).
  \end{align*}
  \endgroup
  Adding
  \[
    \sum_{\substack{s,t\in \partial_c\\s<t}}\sum_{\substack{c<a<b\\a<b<c\\b<c<a}}m_{a_c}(s_{\out})m_{b_c}(t_{\out}),
  \]
  we obtain
  \begingroup
  \allowdisplaybreaks
  \begin{align*}
    &\sum_{\substack{a,b\neq c\\a<b}}(ab+b|_{a<c}+a|_{b<c}+1|_{c<a})m_{a_c}(s_{\out})m_{b_c}(t_{\out})\\
    &\quad+ \sum_{\substack{a,b\neq c\\a>b}}(ab+b|_{c<a}+a|_{c<b}+1|_{c<a})m_{a_c}(s_{\out})m_{b_c}(t_{\out})\Bigr)\\
    &=\sum_{a<b}(b|_{a<c} + a|_{b<c} + 1|_{c<a})m_{a_c}(s_{\out})m_{b_c}(t_{\out}) + \sum_{a>b}(b|_{c<a}+a|_{c<b}+1|_{c<a})m_{a_c}(s_{\out})m_{b_c}(t_{\out})\Bigr)\\
    &\quad + \sum_{a<b}ab\bigl(m_{a_c}(s_{\out})m_{b_c}(t_{\out}) + m_{b_c}(s_{\out})m_{a_c}(t_{\out})
  \end{align*}
  \endgroup
  Now sum over all $s,t\in \partial_c(z)$, $s<t$. Adding
  \[
    \sum_{\substack{a,b\neq c\\a<b}}ab\cdot m_{ac}m_{bc},
  \]
  we obtain
  \begingroup
  \allowdisplaybreaks
  \begin{align*}
    &\sum_{\substack{s,t\in \partial_c(z)\\s<t}}\Bigl(\sum_{\substack{a,b\neq c\\a<b}} (b|_{a<c} + a|_{b<c} + 1|_{c<a})m_{a_c}(s_{\out})m_{b_c}(t_{\out})\\
    &\quad+ \sum_{\substack{a,b\neq c\\a>b}} (b|_{c<a}+a|_{c<b}+1|_{c<a})m_{a_c}(s_{\out})m_{b_c}(t_{\out})\Bigr)+ \sum_{s\in \partial_c(z)}\sum_{a<b}ab\cdot m_{a_c}(s_{\out})m_{b_c}(s_{\out})\\
    &=\sum_{\substack{s,t\in \partial_c(z)\\s<t}}\Bigl(\sum_{a<b}(b|_{a<c} + a|_{b<c} + 1|_{c<a})m_{a_c}(s_{\out})m_{b_c}(t_{\out})\\
    &\quad+ \sum_{a>b}(b|_{c<a}+a|_{c<b}+1|_{c<a})m_{a_c}(s_{\out})m_{b_c}(t_{\out})\Bigr)\\
    &\quad + \sum_{s\in \partial_c(z)}\sum_{\substack{a,b\neq c\\a<b}} (b|_{c<a} + a|_{c<b} + 1|_{c<a})m_{a_c}(s_{\out})m_{b_c}(s_{\out})\mod \coboundary, 
  \end{align*}
  \endgroup
  where the last equality is by Lemma \ref{simple coboundary}. Adding
  \[
    \sum_{\substack{s,t\in \partial_c\\s<t}}\sum_{\substack{c<a<b\\a<b<c\\b<c<a}}(a+b)m_{a_c}(s_{\out})m_{b_c}(t_{\out}) + \sum_{\substack{c,\\s\in \partial_c}}\sum_{\substack{c<a<b\\a<b<c}}(a+b)m_{a_c}(s_{\out})m_{b_c}(s_{\out}),
  \]
  we obtain
  \begin{align*}
    &\sum_{\substack{c,\\s,t\in \partial_c(z)\\s<t}}\Bigl(\sum_{a<b}(a|_{c<a} + b|_{c<b} + 1|_{c<a})m_{a_c}(s_{\out})m_{b_c}(t_{\out})\\
    &\quad+ \sum_{a>b}(a|_{c<a}+b|_{c<b}+1|_{c<a})m_{a_c}(s_{\out})m_{b_c}(t_{\out})\Bigr)\\
    &\quad + \sum_{\substack{c,\\s\in \partial_c(z)}}\sum_{\substack{a,b\neq c,\\a<b}} (a|_{c>a} + b|_{c>b} + 1|_{c<a})m_{a_c}(s_{\out})m_{b_c}(s_{\out})\\
    &= \sum_{\substack{c,\\s,t\in \partial_c(z)\\s<t}}\Bigl(\sum_{a,b\neq c}(a|_{c<a} + b|_{c<b})m_{a_c}(s_{\out})m_{b_c}(t_{\out})\Bigr) +  \sum_{\substack{c,\\s\in \partial_c(z)}}\sum_{\substack{a,b\neq c,\\a<b}} (a|_{c>a} + b|_{c>b})m_{a_c}(s_{\out})m_{b_c}(s_{\out})\\
    &\quad +\sum_{\substack{c,\\s,t\in \partial_c(z)\\s<t}}\Bigl(\sum_{\substack{a,b,c\\ \text{disjoint}}}(1|_{c<a})m_{a_c}(s_{\out})m_{b_c}(t_{\out})\Bigr) +  \sum_{\substack{c\\s\in \partial_c(z)}}\sum_{\substack{a,b\neq c,\\a<b}} (1|_{c<a})m_{a_c}(s_{\out})m_{b_c}(s_{\out})\\
  \end{align*}
  Now using Lemma \ref{simple a+b}, we can reverse the inequalities in the above term $a|_{c>a}+b|_{c>b}$ by adding the (identically zero) term in Equation (\ref{simple a+b equation}). From Lemma \ref{iterate through s,a}, the top two sums add to $0$ mod coboundary. So we are left with the bottom two sums, which equal
  \[
    \sum_{\substack{c,\\s,t\in \partial_c(z)\\s<t}}\Bigl(\sum_{\substack{a,b,c\\\text{disjoint,}\\c<a}}m_{a_c}(s_{\out})m_{b_c}(t_{\out})\Bigr) + \sum_{{\substack{c,\\s\in \partial_c(z)}}}\sum_{c<a<b}m_{a_c}(s_{\out})m_{b_c}(t_{\out}).\qedhere
  \]
\end{proof}
Using Proposition \ref{big simplify}, we are able to simplify the difference formula as follows:
\begin{align*}
  &\langle\sq^2(\alpha) - \mathfrak{sq}^2(\alpha),z\rangle \\
  &=  \sum_{a<b<c<d}(a+b+c+d)m_{ac}m_{bd}+ \sum_{\substack{c,\\s,t\in \partial_c\\s<t}}\sum_{\substack{c<a<b\\a<b<c\\b<c<a}}(a+b)m_{a_c}(s_{\out})m_{b_c}(t_{\out})\\
  &+ \sum_{\substack{c,\\s\in \partial_c}}\sum_{\substack{c<a<b\\a<b<c}}(a+b)m_{a_c}(s_{\out})m_{b_c}(s_{\out}) + \sum_{\substack{c,\\s\in \partial_c}}c\frac{m(s_{\out})}{2} + \sum_{a<b}\binom{m_{ab}}{2}\\
  &+ \sum_{\substack{s,t\in \partial_c\\s<t}}\Bigl(\sum_{\substack{a,b,c\\\text{disjoint,}\\c<a}}m_{a_c}(s_{\out})m_{b_c}(t_{\out})\Bigr) + \sum_{{\substack{c,\\s\in \partial_c}}}\sum_{c<a<b}m_{a_c}(s_{\out})m_{b_c}(t_{\out}) \mod \coboundary.
\end{align*}
We simplify the last three terms in this expression using the following lemma:
\begin{lemma}\label{c<a simplify}
  \begin{equation}\label{c<a}
  \begin{aligned}
    &\sum_{a<b}\binom{m_{ab}}{2} + \sum_{\substack{c,\\s,t\in \partial_c\\s<t}}\Bigl(\sum_{\substack{a,b,c\\\text{disjoint,}\\c<a}}m_{a_c}(s_{\out})m_{b_c}(t_{\out})\Bigr) + \sum_{{\substack{c,\\s\in \partial_c}}}\sum_{c<a<b}m_{a_c}(s_{\out})m_{b_c}(t_{\out})\\
    &= \sum_{\substack{c,\\s\in \partial_c}}\binom{\sum_{c<a} m_{a_c}(s_{\out})}{2}\pmod 2.
  \end{aligned}
  \end{equation}
\end{lemma}
\begin{proof}
  First note that $\binom{m_{ab}}{2}$ counts pairs in $\partial_{ab}$. The following quantity also counts pairs in $\partial_{ab}$:
  \[
    \sum_{\substack{s,t\in \partial_b\\s<t}}m_a(s_{\out})m_a(t_{\out}) + \sum_{s\in \partial_b}\binom{m_a(s_{\out})}{2}.
  \]
  Summing over all $a<b$, we have the identity
  \[
    \sum_{a<b}\binom{m_{ab}}{2} = \sum_{\substack{s,t\in \partial_b\\s<t}}\sum_{a<b}m_a(s_{\out})m_a(t_{\out}) + \sum_{s\in \partial_b}\binom{m_a(s_{\out})}{2}.
  \]
  The left hand side of our equation of interest (\ref{c<a}) thus equals
  \begin{align*}
    &\sum_{\substack{c,\\s,t\in \partial_c\\s<t}}\Bigl(\sum_{\substack{a,b,c\\\text{disjoint,}\\c<a}}m_{a_c}(s_{\out})m_{b_c}(t_{\out})\Bigr) + \sum_{{\substack{c,\\s\in \partial_c}}}\sum_{c<a<b}m_{a_c}(s_{\out})m_{b_c}(t_{\out})\\
    &\quad+ \sum_{\substack{c,\\s,t\in \partial_c\\s<t}}\sum_{c<a}m_{a_c}(s_{\out})m_{a_c}(t_{\out}) + \sum_{\substack{c,\\s\in \partial_c}}\binom{m_a(s_{\out})}{2} \\
    &= \sum_{\substack{c,\\s,t\in \partial_c\\s<t}}\sum_{\substack{a,b\\c<a}}m_{a_c}(s_{\out})m_{b_c}(s_{\out}) + \sum_{{\substack{c,\\s\in \partial_c}}}\binom{\sum_{c<a}m_{a_c}(s_{\out})}{2}\\
    &=\sum_{\substack{c,\\s,t\in \partial_c\\s<t}}\Bigl(\sum_{c<a}m_{a_c}(s_{\out})\underbrace{\sum_b m_{b_c}(s_{\out})}_{=0}\Bigr) +  \sum_{{\substack{c,\\s\in \partial_c}}}\binom{\sum_{c<a}m_{a_c}(s_{\out})}{2}.\qedhere
  \end{align*}
\end{proof}
\begin{proof}[Proof of Theorem \ref{main theorem}]
We obtain from Lemma \ref{c<a simplify} the following simplification:
\begin{equation}
  \begin{aligned}\label{almost done}
  &\langle\sq^2(\alpha) - \mathfrak{sq}^2(\alpha),z\rangle \\
  &=  \sum_{a<b<c<d}(a+b+c+d)m_{ac}m_{bd}+ \sum_{\substack{s,t\in \partial_c\\s<t}}\sum_{\substack{c<a<b\\a<b<c\\b<c<a}}(a+b)m_{a_c}(s_{\out})m_{b_c}(t_{\out})\\
  &+ \sum_{\substack{c,\\s\in \partial_c}}\sum_{\substack{c<a<b\\a<b<c}}(a+b)m_{a_c}(s_{\out})m_{b_c}(s_{\out}) + \sum_{\substack{c,\\s\in \partial_c}}c\frac{m(s_{\out})}{2} +  \sum_{\substack{c,\\t\in \partial_c}}\binom{\sum_{c<a} m_{a_c}(t_{\out})}{2}
    \end{aligned}
\end{equation}
mod coboundary. By Lemma \ref{weighted cross counting}, the right hand side further simplifies to
  \[
    \sum_{\substack{c,\\t\in \partial_c}}\sum_{\substack{a\neq c,\\s\in \partial_{a_c}(t_{\out},\alpha)}}\overleftarrow{\order}(s)\cdot a + \sum_{\substack{c,\\t\in \partial_c}}\binom{\sum_{c<a} m_{a_c}(t_{\out})}{2}\pmod 2.
  \]
  The second summand is equal to the quantity
  \[
    \sum_{\substack{c,\\t\in \partial_c}}\sum_{\substack{a>c\\s\in \partial_{a_c}(t_{\out},\alpha)}}\overleftarrow{\order}(s),
  \]
  simply from the fact that $\binom{k}{2} = 0+1+2+\ldots + (k-1)$. We are left with
  \begin{align*}
    \langle\sq^2(\alpha) - \mathfrak{sq}^2(\alpha),z\rangle
    &= \sum_{\substack{c,\\t\in \partial_c}}\sum_{\substack{a\neq c,\\s\in \partial_{a_c}(t_{\out},\alpha)}}\overleftarrow{\order}(s)\cdot a + \sum_{\substack{c,\\t\in \partial_c}}\sum_{\substack{c<a\\s\in \partial_{a_c}(t_{\out},\alpha)}}\overleftarrow{\order}(s)\\
    &= \sum_{\substack{c,\\t\in \partial_c}}\sum_{\substack{a\neq c,\\s\in \partial_{a_c}(t_{\out},\alpha)}}\overleftarrow{\order}(s)\cdot a_c\\
    &= \sum_{\substack{c,\\t\in \partial_c}}\sum_{\substack{a,\\s\in \partial_{a}(t_{\out},\alpha)}}\overleftarrow{\order}(s)\cdot a = 0\mod\coboundary.\qedhere    
  \end{align*}
\end{proof}
\begin{proof}[Proof of Corollary \ref{main corollary}]
  The corollary follows by applying Theorem \ref{main theorem} to the functor $F = F_{Kh}(L)$.
\end{proof}

\appendix
\section{Basic cochain identities}\label{appendix}
We fix an element $z\in X_{n+2}$ and a cocycle $\alpha\in C^n(X_\bullet;\mathbb{F}_2)$ in this section.
\begin{lemma}\label{m_ab count}
  Fix an $a\in \mathbb{N}$. We have the following identity:
  \[
    \sum_{\substack{b,\\b\neq a}}m_{ab}(z,\alpha) = 0\pmod 2.
  \]
  
\end{lemma}
\begin{proof}
   \[
    \sum_{\substack{b,\\b\neq a}}m_{ab}(z,\alpha) = \sum_{s\in \partial_a(z)}\sum_{b\neq a}\#\partial_{b_a}(s_{\out},\alpha) = \sum_{s\in \partial_a(z)}\#\partial(s_{\out},\alpha) =  \sum_{s\in \partial_a(z)} 0\pmod 2.\qedhere
  \]
\end{proof}

\begin{proposition}[compare \cite{rajapakse2025}]\label{a sum}
  Define $L: C^*(X_\bullet;\mathbb{F}_2)\to C^{*+2}(X_\bullet;\mathbf{F}_2)$ by
  \[
    \langle L\omega, z \rangle = \sum_{a<b}b m_{ab}(z,\omega).
  \]
  The linear map $L$ is nullhomotopic. In particular,
  \[
    \sum_{a<b}bm_{ab}(z) = 0 \mod\coboundary
  \]
\end{proposition}

  \begin{proof}
    Define a homotopy $H: C^*(X_\bullet)\to C^*(X_\bullet)^{*+1}$ by
    \[
      \langle H\omega,y \rangle = \sum_{a} \binom{a+1}{2}m_a(y,\omega)
    \]
    We have $\langle\delta Hx + H\delta x, z\rangle$ contributing four terms per element $s\in \partial_{ab}(z,x)$, as seen in Figure \ref{homotopy}.
    \def\T{0.3cm}
    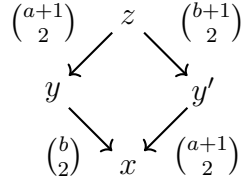
\begin{figure}
    \begin{tikzpicture}[scale=1]
  \path[draw, <-, shorten <=\T,shorten >=\T, thick] (0,1) -- (1,2) node [midway, label={[label distance=-0.2cm]135:$\binom{a+1}{2}$}] {};
  \path[draw, <-, shorten <=\T,shorten >=\T, thick] (2,1) -- (1,2) node [midway, label={[label distance=-0.2cm]45:$\binom{b+1}{2}$}] {};
  \path[draw, <-, shorten <=\T,shorten >=\T, thick] (1,0) -- (2,1) node [midway, label={[label distance=-0.3cm]-45:$\binom{a+1}{2}$}] {};
  \path[draw, <-, shorten <=\T,shorten >=\T, thick] (1,0) -- (0,1) node [midway, label={[label distance=-0.3cm]225:$\binom{b}{2}$}] {};
     \node at (1, 2) {$z$};
   \node at (0, 1) {$y$};
   \node at (2, 1) {$y'$};
   \node at (1, 0) {$x$};
 \end{tikzpicture}
 \caption{We can write a span element $s:z\to x$ either as a composition in $\partial_b$.}
 \label{homotopy}
 \end{figure}
    \[
      \langle \delta Hx + H\delta x, z\rangle = \sum_{s\in \partial(z,x)}\Bigl(\binom{b}{2} + \binom{a+1}{2} + \binom{b+1}{2} + \binom{a+1}{2}\Bigr) = \sum_{s\in \partial(z,x)} b\mod 2,
    \]
    which is equal to $\langle Lx, z\rangle$.
  \end{proof}
  \begin{lemma}\label{iterate through s,a}
  \begin{align*}
    &\sum_{\substack{s,t\in \partial_c\\s<t}} \sum_{\substack{a,b,c\\ \text{disjoint}}}(a|_{c<a}+b|_{c<b})m_{a_c}(s_{\out})m_{b_c}(t_{\out}) + \sum_{\substack{c,\\s\in \partial_c}}\sum_{\substack{a,b\neq c\\a<b}}(a|_{c< a}+b|_{c< b})m_{a_c}(s_{\out})m_{b_c}(s_{\out}) = 0\\
    &\mod \coboundary.
  \end{align*}  
\end{lemma}
\begin{proof}
  If we add the term
  \[
    \sum_{{\substack{c,\\s\in \partial_c}} }\sum_{b}b|_{c<b}m_{b_c}(s_{\out}),
  \]
  which is $0$ mod coboundary (see Proposition \ref{a sum}), to the left hand side, we obtain
  \begin{align*}
    &\sum_c\Bigl(\sum_{{\substack{c,\\s\in \partial_c}}(z)}\sum_{a\neq c} a|_{c< a}m_{a_c}(s_{\out})\Bigr)\Bigl(\sum_{t\in \partial_c(z)}\underbrace{\sum_{b\neq c} m_{b_c}(t_{\out})}_{=0}\Bigr) = 0.\qedhere
  \end{align*}
\end{proof}
\begin{lemma}\label{simple coboundary}
  \begin{equation}\label{ab sum}
    \sum_{\substack{c,\\s\in \partial_c(z)}}\sum_{\substack{a,b\neq c\\a<b}}ab\cdot m_{a_c}(s_{\out})m_{b_c}(s_{\out}) = \sum_{\substack{c,\\s\in \partial_c(z)}}\sum_{\substack{a,b\neq c\\a<b}} (b|_{c<a} + a|_{c<b} + 1|_{c<a})m_{a_c}m_{b_c}
  \end{equation}
\end{lemma}
\begin{proof}
  Note first that $ab = a_cb_c + b|_{c<a} + a|_{c<b} + 1|_{c<a}$. Our lemma is complete once we observe the difference of the two terms in Equation (\ref{ab sum}) is
  \[
    \sum_{\substack{c,\\s\in \partial_c(z)}}\sum_{\substack{a,b\neq c\\a<b}}a_cb_c\cdot m_{a_c}(s_{\out})m_{b_c}(s_{\out})=\sum_{\substack{c,\\s\in \partial_c(z)}}\sum_{a<b}ab\cdot m_{a}(s_{\out})m_{b}(s_{\out}),
  \]
  and the final term is $0$ mod coboundary.
\end{proof}

\begin{lemma}\label{fa+fb}
  Let $f:\mathbb{Z}\to \mathbb{Z}$ be any function. We have
  \[
    \sum_{a<b}(f(a)+f(b))m_{ab}= 0 \pmod 2.
  \]
  
\end{lemma}
\begin{proof}
  We have
  \[
    \sum_{a<b}(f(a)+f(b))m_{ab} = \sum_af(a)\Bigl(\sum_{b\neq a}m_{ab}\Bigr) = \sum_a f(a) \cdot 0 \mod 2,
  \]
  with the last equality being from Lemma \ref{m_ab count}.
\end{proof}
\begin{lemma}\label{simple a+b}
  We have
  \begin{equation}\label{simple a+b equation}
    \sum_{\substack{c,\\s\in \partial_c(z)}}\sum_{a<b}(a+b)m_{a_c}m_{b_c} = 0 \pmod 2
  \end{equation}
\end{lemma}
\begin{proof}
  We write
  \begin{align*}
    \sum_{\substack{c,\\s\in \partial_c(z)}}\sum_{a<b}(a+b)m_{a_c}m_{b_c}
    &= \sum_{{\substack{c,\\s\in \partial_c(z)}}}\Bigl(\Bigl(\sum_{a\neq c} a m_{a_c}(s_{\out})\Bigr)\Bigl(\underbrace{\sum_{a\neq c} m_{a_c}(s_{\out})\Bigr)}_{=0\pmod 2}-\sum_{a\neq c} am_{a_c}(s_{\out})\Bigr)\\
    &= \sum_{\substack{c,\\s\in \partial_c(z)}}\sum_{a\neq c} am_{a_c}(s_{\out})\\
    &=\sum_{a<b}(a+b)m_{ab} = 0 \pmod 2,
  \end{align*}
  where the last equality is by Lemma \ref{fa+fb}.
\end{proof}
\begin{lemma}\label{extreme c}
  We have
  \[
    \sum_{\substack{c,\\s\in \partial_c(z)}}\sum_{\substack{c<a<b\\ a<b<c}}m_{a_c}(s_{\out})m_{b_c}(s_{\out}) = 0\mod \coboundary.
  \]
  
\end{lemma}
\begin{proof}
  The left hand side can be rewritten as
  \begin{align*}
    &\sum_{\substack{c,\\s\in \partial_c(z)}}\sum_{\substack{a<b\\ a,b\neq c}}(1|_{c<a} + 1|_{c<b})m_{a_c}(s_{\out})m_{b_c}(s_{\out}) + \sum_{\substack{c,\\s\in \partial_c(z)}}\sum_{\substack{a<b\\ a,b\neq c}}m_{a_c}(s_{\out})m_{b_c}(s_{\out})
  \end{align*}
  Now we add the term
  \[
    \sum_{\substack{c,\\s\in \partial_c(z)}}\sum_{\substack{a<b\\ a,b\neq c}} (1 +a_c + b_c)m_{a_c}(s_{\out})m_{b_c}(s_{\out}),
  \]
  which is equal to $0$ mod coboundary, to obtain
  \begin{align*}
    &\sum_{\substack{c,\\s\in \partial_c(z)}}\sum_{\substack{a<b\\ a,b\neq c}}(1|_{c<a} + 1|_{c<b})m_{a_c}(s_{\out})m_{b_c}(s_{\out}) + \sum_{\substack{c,\\s\in \partial_c(z)}}\sum_{\substack{a<b\\ a,b\neq c}} (a_c + b_c)m_{a_c}(s_{\out})m_{b_c}(s_{\out})\\
    &= \sum_{\substack{c,\\s\in \partial_c(z)}}\sum_{\substack{a<b\\ a,b\neq c}} (a + b)m_{a_c}(s_{\out})m_{b_c}(s_{\out}) = 0\pmod 2
  \end{align*}
  where the last identity is by Lemma \ref{simple a+b}.
\end{proof}

\begin{lemma}\label{Leibniz}
  For all integers $n,a$, we have $\binom{na}{2} = a\binom{n}{2}+ n\binom{a}{2}\pmod 2$.
\end{lemma}
\begin{proof}
  The above identity follows either by induction on $n$, or by explicit computation.
\end{proof}

\begin{lemma}\label{one intersection}
  We have
  \[
    \sum_c c\sum_{\substack{a,b\neq c\\a<b}}(a+b) m_{ac}m_{bc}= 0\pmod 2.
  \]
  
\end{lemma}
\begin{proof}
  The left hand side can be written as
  \begin{align*}
    &\sum_c c \Biggl(\Bigl(\sum_{a\neq c}a\cdot m_{ac}\Bigr)\underbrace{\Bigl(\sum_{b\neq c}m_{bc}\Bigr)}_{=0}-\sum_{a\neq c} a \cdot m_{ac}\Biggr)\\
    &= \sum_c c \sum_{a\neq c}a\cdot m_{ac}\\
    &= \sum_{\substack{(a,c)\\a\neq c}}ac\cdot m_{ac}= 0\pmod 2.
  \end{align*}
  The last congruence is shown by observing each $(a,c)$ summand is equal to the $(c,a)$ summand.
\end{proof}
\section{Advanced cochain identities}\label{correspondence ordering}
We introduce a convenient way of drawing the span $\partial(z,\alpha)$, which will help us simplify our formula for $\mathfrak{sq}^2(\alpha)$.\par
Let $\mathbb{H}$ denote the complex upper half-plane $\{z\in \mathbb{C}: \Ima z>0\}$. We call a semicircle in $\mathbb{H}$ centered on the real axis a \textit{chord}. In other words, chords are precisely the lines that do not reach infinity when viewing $\mathbb{H}$ as Poincar\'e's half-plane model.
\begin{definition}
  Given an augmented semi-simplicial object $X_\bullet$, an element $z\in X_{n+2}$, and a set of elements $S\subset X_n$,  we define the \textit{degenerate chord presentation} $\overline{\mathcal{C}}(z,S)$ as the following set of chords $C$ embedded in $\mathbb{H}$: We view each $s\in \partial_{ab}(z,S)$ as a copy of the chord $C_s$ having ends on $a$ and $b$.\par
  We now define the \textit{(nondegenerate) chord presentation} $\mathcal{C}(z,S)$ as a similar presentation of chords, but we perturb the endpoints so no two endpoints meet at the same point (although their index stays the same. (See Figure \ref{degen vs. nondegen}).
  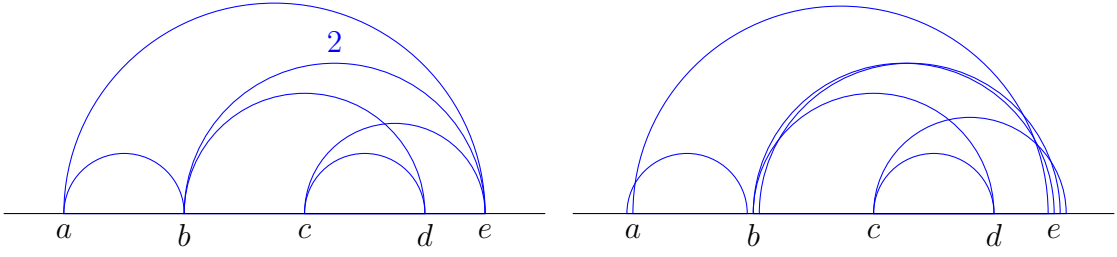
\begin{figure}
      \begin{tabular}{ c c }
      \begin{tikzpicture}[scale=0.8]
        \hpline{1}{3}
        \hpline[$2$]{3}{8}
        \hpline{5}{8}
        \hpline{1}{8}
        \hpline{3}{7}
        \hpline{5}{7}
        \draw (0,0) -- (9,0);
        \coordinate (1) at (1,0);
        \fill (1) node[below] {$a$};
        \coordinate (2) at (3,0);
        \fill (2) node[below] {$b$};
        \coordinate (3) at (5,0);
        \fill (3) node[below] {$c$};
        \coordinate (4) at (7,0);
        \fill (4) node[below] {$d$};
        \coordinate (5) at (8,0);
        \fill (5) node[below] {$e$};
      \end{tikzpicture}
        &
          \begin{tikzpicture}[scale=.8]
        \hpline{0.9}{2.9}
        \hpline{3}{8}
        \hpline{3.1}{8.1}
        \hpline{5}{8.2}
        \hpline{1}{7.9}
        \hpline{3}{7}
        \hpline{5}{7}
        \draw (0,0) -- (9,0);
        \coordinate (1) at (1,0);
        \fill (1) node[below] {$a$};
        \coordinate (2) at (3,0);
        \fill (2) node[below] {$b$};
        \coordinate (3) at (5,0);
        \fill (3) node[below] {$c$};
        \coordinate (4) at (7,0);
        \fill (4) node[below] {$d$};
        \coordinate (5) at (8,0);
        \fill (5) node[below] {$e$};
      \end{tikzpicture}
      \end{tabular}
    \caption{Left: A degenerate chord presentation $\overline{\mathcal{C}}(z,\alpha)$, for $\alpha$ a cocycle. There are two chords between $b$ and $e$, so the chord is labeled with a ``$2$'' to show multiplicity. Right: A chord presentation $\mathcal{C}(z,\alpha)$. Note the even number of chords coming out of each index.}
    \label{degen vs. nondegen}
    \end{figure}
    We choose a specific perturbation, that relies on a lexicographic order on a product of spans. For each index $a$, we perturb the chords meeting $a$ by the following procedure:
    \begin{enumerate}[label = (P-\arabic*), ref = (P-\arabic*)]
    \item We write the chords that intersect $a$ as the union
      \begin{equation}\label{chords intersecting a}
        \bigsqcup_{s\in \partial_a(z)}\partial(s_{\out},\alpha)\times \{s\}\subset \partial^{n+1}\times_{X_{n+1}}\partial^{n+2}.
      \end{equation}
    \item Order the above union according to the lexicographic order of $\partial^{n+2}\times \partial^{n+1}$, where the ordering of $\partial^{n+2}, \partial^{n+1}$ is by Definition \ref{ordering convention} (notice the switch of factors here). In particular, we have
      \begin{equation}\label{subindex}
        \partial(s_{\out},\alpha)\times \{s\} <  \partial(t_{\out},\alpha)\times \{t\} \qquad \text{if $s<t$}.
      \end{equation}
    \item Perturb the chords meeting $a$ according to this ordering. That is, higher ordered chords should be to the right of lower ordered chords.
    \end{enumerate}\par
    The disjoint union in (\ref{chords intersecting a}) can be viewed as a partition on the set of chord ends on $a$. And by (\ref{subindex}), each partition occupies its own (small) arc in $S^1$. We say that chords in $\partial(s_{\out},\alpha)\circ \{s\}$  the \textit{meet $s$}.
  \end{definition}
  \begin{remark}
    By picking a biholomorphism $\mathbb{H}\to \mathbb{D}$, we can view degenerate and nondegenerate chord presentations as lying inside $\mathbb{D}$.
  \end{remark}
  
We now prove some (progressively more challenging) lemmas, which will help us gain familiarity with chord presentations.
\begin{lemma}\label{even chords}
  Let $\alpha$ be a cocycle in $C^n(X_\bullet;\mathbb{F}_2)$, and let $z\in X_{n+2}$. Both $\overline{\mathcal{C}}(z,\alpha)$ and $\mathcal{C}(z,\alpha)$ have an even number of chords coming out of each index $a$. Furthermore, both chord presentations have an even number of chords meeting $s\in \partial_a$. 
\end{lemma}
We remark that the proof of this lemma is identical to the proof of Lemma \ref{m_ab count}, but stated in the language of chord presentations.
\begin{proof}
  Fix an $s\in \partial_a$. The number of chords meeting $s$ is equal to $\#\partial(s_{\out},\alpha)$, which is $0$ mod $2$ by virtue of $\alpha$ being a cocycle. Adding up over all $s\in \partial_a(z)$, we obtain an even number of chords coming out of $a$.
\end{proof}
\begin{lemma}\label{ab+cd lemma}
  We have
  \begin{equation}\label{ab+cd=0}
    \sum_{a<b<c<d} (ab+cd) (m_{bc}m_{ad} + m_{cd}m_{ab} + m_{ac}m_{bd}) = 0.
  \end{equation}
\end{lemma}
\begin{proof}
  We rewrite the sum as the following:
  \begin{align}\label{ab+cd rewritten}
    &\sum_{a<b}ab\sum_{\substack{c,d\\a<c<b<d\text{ or}\\c<a<d<b}}(m_{ab}m_{cd} + m_{ad}m_{bc} + m_{ac}m_{bd}) := \sum_{a<b} ab \cdot S_{ab}
  \end{align}
  It remains to prove that the interior sum $S_{ab}$ is $0$ for all $a<b$. $S_{ab}$ is the total count of three types of pairs of chords in $\mathcal{C}(z,\alpha)$ viewed, this time, on the unit disk $\mathbb{D}$ (see Figure \ref{chord pairs}):
  \begin{enumerate}[label = (C-\arabic*), ref = (C-\arabic*)]
  \item $(C_{ab},C_{cd})$, where $C_{ab}\in \partial_{ab}(z)$, and $C_{cd}$ intersects $C_{ab}$ transversely.\label{type 1}
  \item $(C_{ac}, C_{bd})$, where $C_{ac}$ veers counterclockwise out of $a$ in the direction of $b$, and $C_{bd}$ veers counterclockwise out of $b$ in the direction of $a$.\label{type 2r}
  \item $(C_{ac}, C_{bd})$, where $C_{ac}$ veers clockwise out of $a$ in the direction of $b$, and $C_{bd}$ veers clockwise out of $b$ in the direction of $a$.\label{type 2l}
  \end{enumerate}
  \begin{figure}
    \begin{tabular}{ c c c }
      \begin{tikzpicture}[scale=2]
  \draw (0,0) circle (1);
  \begin{scope}
    \clip (0,0) circle (1);
    \hgline{120}{290}
    \hgline{60}{190}
   \end{scope}
    \coordinate (B) at (-70:1);
    \fill (B) circle (1pt) node[below right] {$b$};
    \coordinate (A) at (120:1);
    \fill (A) circle (1pt) node[above left] {$a$};
  \end{tikzpicture} &
  \begin{tikzpicture}[scale=2]
  \draw (0,0) circle (1);
  \begin{scope}
    \clip (0,0) circle (1);
    \hgline{-70}{20}
    \hgline{120}{190}
   \end{scope}
    \coordinate (B) at (-70:1);
    \fill (B) circle (1pt) node[below right] {$b$};
    \coordinate (A) at (120:1);
    \fill (A) circle (1pt) node[above left] {$a$};
  \end{tikzpicture} &
  \begin{tikzpicture}[scale=2]
  \draw (0,0) circle (1);
  \begin{scope}
    \clip (0,0) circle (1);
    \hgline{290}{200}
    \hgline{120}{60}
   \end{scope}
    \coordinate (B) at (-70:1);
    \fill (B) circle (1pt) node[below right] {$b$};
    \coordinate (A) at (120:1);
    \fill (A) circle (1pt) node[above left] {$a$};
  \end{tikzpicture}
      \end{tabular}
  \caption{The types of chord pairs in $\overline{\mathcal{C}}(z,\alpha)$ counted in the interior summand of Equation (\ref{ab+cd rewritten}). Left: a chord pair of type \ref{type 1}. Middle: a chord pair of type \ref{type 2r}. Right: a chord pair of type \ref{type 2l}.}
  \label{chord pairs}
\end{figure}
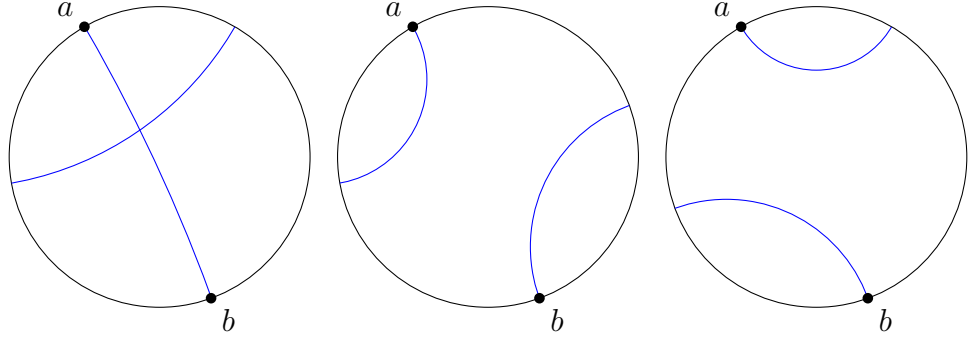
The number of chord pairs of type \ref{type 1} is $m_{ab}x_{ab}$: the number $m_{ab}$ of chords joining $a$ to $b$ multiplied by the number $x_{ab}$ of chords that separate $a$ and $b$. The number of pairs of type \ref{type 2r} is $r_ar_b$, where $r_a$ is the number of chords veering counterclockwise, or ``right,'' out of $a$ in the direction of $b$, and $r_b$ is defined similarly, with the roles of $a$ and $b$ reversed. Similarly, the number of pairs of type \ref{type 2l} is $l_al_b$, where $l_a$, $l_b$ are defined similar to $r_a$, $r_b$, but with ``counterclockwise'' replaced with ``clockwise.'' The total number of chords coming out of $a$ is equal to $m_{ab} + l_a+r_a$, which is $0$ because $\alpha$ is a cocycle (see Lemma \ref{even chords}). We also observe that $x_{ab}= l_a+r_b\mod 2$, again using the fact that the degree of every vertex is even. Putting these identities together, we find
  \begin{align*}
      S_{ab} & = m_{ab}x_{ab} + l_al_b + r_ar_b\\
             &= (l_a+r_a)(l_a+r_b) + l_al_b + r_ar_b\\
             &= l_al_b+l_a+l_ar_b+l_ar_a\\
             &= l_a(1+l_b+r_b+r_a)\\
      &= l_a(1+l_a) = 0\pmod 2.\qedhere
    \end{align*}
  \end{proof}

  \begin{proposition}\label{cross counting}
    \begin{equation}\label{cross counts}
    \begin{aligned}
      &\sum_{a<b<c<d}m_{ac}(z)m_{bd}(z) + \sum_{\substack{s,t\in \partial_c\\s<t}}\sum_{\substack{c<a<b\\a<b<c\\b<c<a}}m_{a_c}(s_{\out})m_{b_c}(t_{\out})\\
      &\quad+ \sum_{a<b} \#\left\{\{s,t\}\subset\partial_{ab}(z,\alpha)\Bigm|\ \text{left-break order of $\{s,t\}$ agrees with right-break order}\right\}\\
    &= \sum_{K\subset \Gamma(z,\alpha)} \bigl(1 + \#\{a\to \overrightarrow{b}\to c\ \vert\  a>b\} + \#\{a\to \overleftarrow{b}\to c\ \vert\  a<b\}\bigr)\mod \coboundary.
    \end{aligned}
    \end{equation}
  \end{proposition}
  \begin{proof}
    Adding
    \[
       \sum_{s\in \partial_c}\sum_{\substack{c<a<b\\a<b<c}}m_{a_c}(s_{\out})m_{b_c}(t_{\out})
    \]
    to the left-hand side of Equation (\ref{cross counts}), we obtain a term that simply counts the number of pairs of chords in $\mathcal{C}(z,\alpha)$ that cross. But the term we added is $0$ mod coboundary from Proposition \ref{extreme c}. It remains to show that the right hand side is equal (mod 2) to the number of crosses. \par
    The cycles $K\subset \Gamma(z,\alpha)$ partition the set $\mathcal{C}(z,\alpha)$ into subsets corresponding to each $K$. Indeed, each edge $e_1\xline{e'} e_2$ corresponds with a unique chord $\mathcal{C}(e')$ and thus each cycle $K$ corresponds to a cycle of chords $\mathcal{C}(K)\subset \mathcal{C}(z,\alpha)$ (see Figure \ref{cycle crossings}).
    \begin{figure}
      \begin{tikzpicture}[scale=1]
        \hpline{8.1}{13}
        \hpline{3}{5.7}
        \hpline{5.8}{8.0}
        \hpline{6.1}{12.9}
        \hpline{2.9}{6.0}
        \hplinesty[densely dashed]{1.0}{3.2}
        \hplinesty[densely dashed]{3.3}{8.3}
        \hplinesty[densely dashed]{8.4}{13.2}
        \hplinesty[densely dashed]{1.1}{10}
        \hplinesty[densely dashed]{10.1}{13.3}
        \draw (0,0) -- (15,0);
        \coordinate (1) at (1,0);
        \fill (1) node[below] {$a_1$};
        \coordinate (2) at (3,0);
        \fill (2) node[below] {$a_2$};
        \coordinate (3) at (6,0);
        \fill (3) node[below] {$a_3$};
        \coordinate (4) at (8.1,0);
        \fill (4) node[below] {$a_4$};
        \coordinate (5) at (10.1,0);
        \fill (5) node[below] {$a_5$};
        \coordinate (6) at (13,0);
        \fill (6) node[below] {$a_5$};
      \end{tikzpicture}
  \caption{We draw two cycles of chords (one solid, one dotted) corresponding to two graph cycles $K,K'\subset \Gamma(z,\alpha)$. Note how these two cycles have even intersection number, as they always should. We compute $\#(\mathcal{C}(K)\cap\mathcal{C}(K'))$ by adding the internal crossings in each cycle.}
  \label{cycle crossings}
    \end{figure}
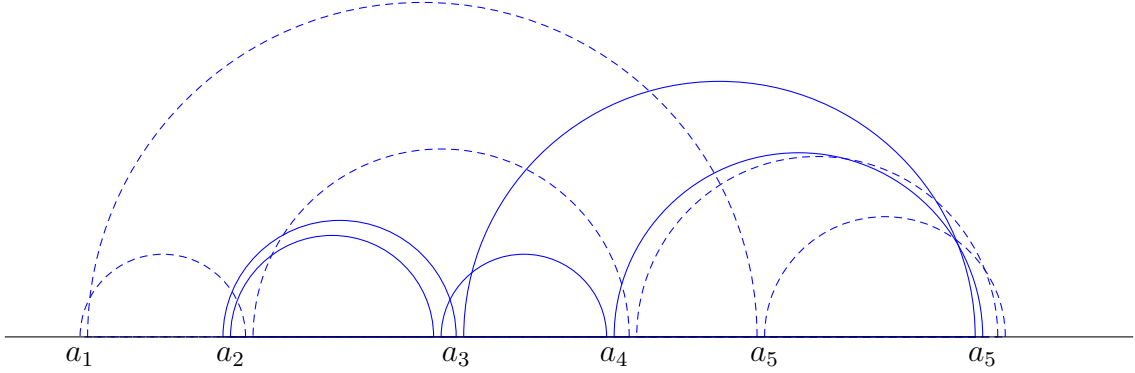
    For any two separate chord cycles $\mathcal{C}(K)$, $\mathcal{C}(K')$, we can smooth out their corners and make them closed embedded loops, which have even intersection number. Therefore, the number of crosses in $\mathcal{C}(z,\alpha)$ is computed by adding the internal number of crosses in each $\mathcal{C}(K)$:
    \[
      \#(\text{crossings in $\mathcal{C}(z,\alpha)$}) = \sum_{K\subset \Gamma(z,\alpha)} \#(\text{crossings in $\mathcal{C}(K)$})\pmod 2.
    \]
    Therefore, it suffices to prove
    \[
      \#(\text{crossings in $\mathcal{C}(K)$}) = \#\{a\to \overrightarrow{b}\to c\ \vert\  a>b\} + \#\{a\to \overleftarrow{b}\to c\ \vert\  a<b\}
    \]
    for any cycle $K\subset\Gamma(z,\alpha)$
    \[
      e_1 \xline{e'_1} e_2 \xline{e'_2} \ldots \xline{e'_3} e_k\xline{e'_k} e_1,
    \]
    with facet cycle oriented as
    \[
      a_1\to a_2\to \ldots\to a_m\to a_1.
    \]
    Now fix a cycle $K$ and view the chords of $\mathcal{C}(K)$ in the upper half-plane (see Figure \ref{intersection figure}).
    \begin{figure}
      \begin{tikzpicture}[scale=1]
        \hpline[$e'_2$]{8.1}{10}
        \hpline[$e'_1$]{3}{8}
        \hpline[$e'_3$]{5}{9.9}
        \draw (0,0) -- (15,0);
        \coordinate (1) at (3,0);
        \fill (1) node[below] {$a_1$};
        \coordinate (2) at (8,0);
        \fill (2) node[below] {$a_2$};
        \coordinate (3) at (10,0);
        \fill (3) node[below] {$a_3$};
        \coordinate (4) at (5,0);
        \fill (4) circle(2pt) node[below] {$a_4$};
        \hplinesty[densely dashed]{5.1}{12}
        \hplinesty[densely dotted]{4.9}{7}
      \end{tikzpicture}
      \caption{We count the quantity $X(k)$ in (\ref{telescope}) for $k=4$. If $e'_4$ was the dashed chord, then $X(4) = 2 + 0 = 0\pmod 2$, since $a_{k+1}>a_k$ but $a_{k-1}\to \overrightarrow{a_{k}}\to a_{k+1}$. If $e'_4$ was the dotted chord, then $X(4) = 1 + 1 = 0\pmod 2$. The dot highlights the $a_k$ endpoint of $e'_{k-1}$.}
      \label{intersection figure}
    \end{figure}
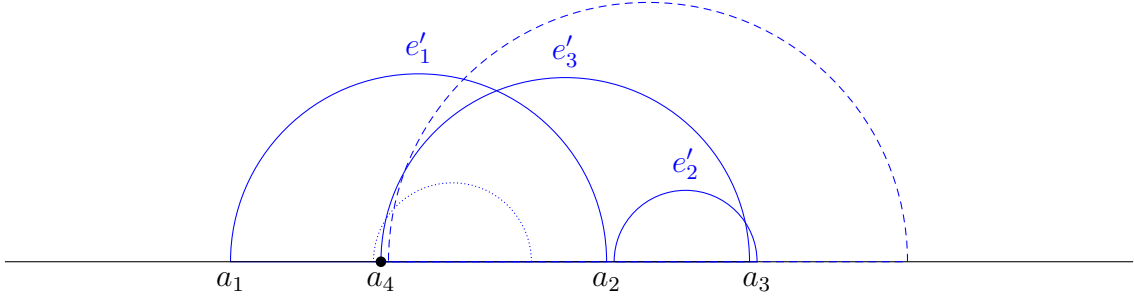
    By reparametrizing our cycle if necesssary, we can assume that the left end of $\mathcal{C}(e'_1)$ lies on $a_1$ and is the leftmost end of any chord in $\mathcal{C}(K)$. Since
    \begin{align*}
      &\#\{a\to \overrightarrow{b}\to c\ \vert\  a>b\} + \#\{a\to \overleftarrow{b}\to c\ \vert\  a<b\} \\
      &= \#\{a\to \overrightarrow{b}\to c\ \vert\  c<b\} + \#\{a\to \overleftarrow{b}\to c\ \vert\  c>b\}\pmod 2,
    \end{align*}
    it suffices to prove
    \begin{equation}\label{crossing identity}
      \#(\text{crossings in $\mathcal{C}(K)$}) =  1 + \#\{a\to \overrightarrow{b}\to c\ \vert\  c<b\} + \#\{a\to \overleftarrow{b}\to c\ \vert\  c>b\}.
    \end{equation}
    To prove this identity, consider the quantity
    \begin{equation}\label{telescope}
      X(k) = \#\{\text{$\mathcal{C}(e'_k)\cap \mathcal{C}(e'_j):1\leq j<k$}\}  + \begin{cases*}
        1 & if $a_{k-1}\to \overrightarrow{a_{k}}\to a_{k+1}$ and $a_{k+1}< a_k$\\
        1 & if $a_{k-1}\to \overleftarrow{a_{k}}\to a_{k+1}$ and $a_{k+1} > a_k$\\
        0  & otherwise,
      \end{cases*}
    \end{equation}
    where in the second term, we define $a_{0} := a_m$, $a_{m+1} := a_{1}$.
    By summing these terms from $1$ to $m$, we obtain all the terms in (\ref{crossing identity}) except the term that is $1$. For all $k$ where $1<k<m$, the quantity $X(k) = 0$. Indeed, if either of the first two conditions in Equation (\ref{telescope}) hold, then the chord $\mathcal{C}(e'_k)$ must have endpoints to the right and left of the $a_k$-endpoint of $\mathcal{C}(e'_{k-1})$. Therefore, the intersection number of $\mathcal{C}(e'_k)$ with the chords $\mathcal{C}(e'_1)\cup\ldots\cup\mathcal{C}(e'_{k-1})$ is $1$. It remains to show that $X(1)+X(m) = 1$. To this end, observe that $X(1) = 0 + 1 = 1$, because $a_m\to \overleftarrow{a_1}\to a_2$ and $a_2>a_1$. Now counting the intersection number of $\mathcal{C}(e'_m)$ with $\mathcal{C}(e'_1)\cup\ldots\cup\mathcal{C}(e'_{m-1})$, we determine 
    \[
      \#\{\text{$\mathcal{C}(e'_m)\cap \mathcal{C}(e'_j):1\leq j<m$}\} =
      \begin{cases*}
        1 & if $a_{m-2}\to \overrightarrow{a_{m-1}}\to a_m$\\
        0 & otherwise.
      \end{cases*}
    \]
    This is equal to the second summand of (\ref{telescope}), as $a_{m+1}=a_1<a_m$. Therefore, $X(m) = 0$.
  \end{proof}
  Before the next proposition, we give a necessary definition.
  \begin{definition}
    Let $y\in X_{n+1}$. We define a total ordering $<$ on the set $\partial(y,\alpha)$ to be the order induced by our total ordering on $\partial^{n+1}$. Then we define
    \[
      \order(s) = \#\{t\in \partial(y,\alpha): t<s\},\qquad \overleftarrow{\order}(s) = \#\{t\in \partial(y,\alpha): t>s\}.
    \]
  \end{definition}
  \begin{proposition}\label{weighted cross counting}
  We have
  \begin{align*}
    &\sum_{a<b<c<d}(a+b+c+d)m_{ac}m_{bd} + \sum_{\substack{s,t\in \partial_c\\s<t}}\sum_{\substack{c< a<b\\a<b<c\\b<c<a}}(a+b)m_{a_c}(s_{\out})m_{b_c}(t_{\out}) \\
    &+ \sum_{\substack{c,\\s\in \partial_c}}\sum_{\substack{c< a<b\\a<b<c}}(a+b)m_{a_c}(s_{\out})m_{b_c}(s_{\out}) = \sum_{\substack{a,\\t\in \partial_a}} a\frac{m(t_{\out})}{2} + \sum_{\substack{c,\\t\in \partial_b}}\sum_{\substack{a,\\s\in \partial_{a_b}(t_{\out},\alpha)}}\overleftarrow{\order(s)}\cdot a\pmod 2
  \end{align*}
\end{proposition}
\begin{proof}
  We give an interpretation of the quantity the left hand side computes: Summing over each pair of crossing chords $C,C'\in \mathcal{C}(z,\alpha)$ with ends $a,c$ and $b,d$, we add $a+b+c+d$. The first summand on the left hand side of accounts for the crossings $C\cap C'$ where the ends of $C,C'$ lie in four distinct indices. The next two summands account for the crossings where the ends of $C,C'$ only share one index. Finally, there is no summand accounting for crossings where $C,C'$ share identically indexed ends, since such a contibution is $0$ mod 2 anyway.\par
  We count this weighted sum of crossings in a different way to derive the right side. For each $a$ and each chord $C_a$ with an end in $a$, we traverse along $C_a$ and add the quantity $a$ for each chord that $C$ crosses. Summing over all $a$ and all $C_a$, we obtain
  \[
    \sum_a\sum_{\substack{\text{chords }C\\ \text{meeting }a}}\#\{C':C'\text{ crosses }C\}\cdot a,
  \]
  which, by our discussion, still equals the left hand side. Now we prove this quantity is equal to the right hand side of the above equation. It suffices to fix $a$ and prove
  \begin{equation}\label{crossing count}
     \sum_{\substack{\text{chords }C\\ \text{meeting }a}}\#\{C':C'\text{ crosses }C\} = \sum_{t\in \partial_a} \frac{m(t_{\out})}{2} + \sum_{\substack{b,\\t\in \partial_b}}\sum_{s\in \partial_{a_b}(t_{\out},\alpha)}\overleftarrow{\order}(s)\pmod 2.
   \end{equation}
   Fix a chord $C$ meeting $a$ and $b$ (see Figure \ref{crossing count fig} for a picture). Now we compute the intersection number in (\ref{crossing count}). Denote the ends of $C$ by $s\circ t\in \partial_{b_a}\circ \partial_a$ and $s'\circ t'\in \partial_{a_b}\circ \partial_b$. The intersection number is, mod 2, the sum of the following two quantities:
   \begin{enumerate}[label = (c-\arabic*), ref = c-\arabic*]
   \item The number of chord ends that are to the left of $s\circ t$.\label{s}
   \item The number of chord ends that are to the right of $s'\circ t'$.\label{s'}
   \end{enumerate}
   By Lemma \ref{even chords}, (\ref{s}) $= \order(s)$, and (\ref{s'}) $= \overleftarrow{\order}(s')$
       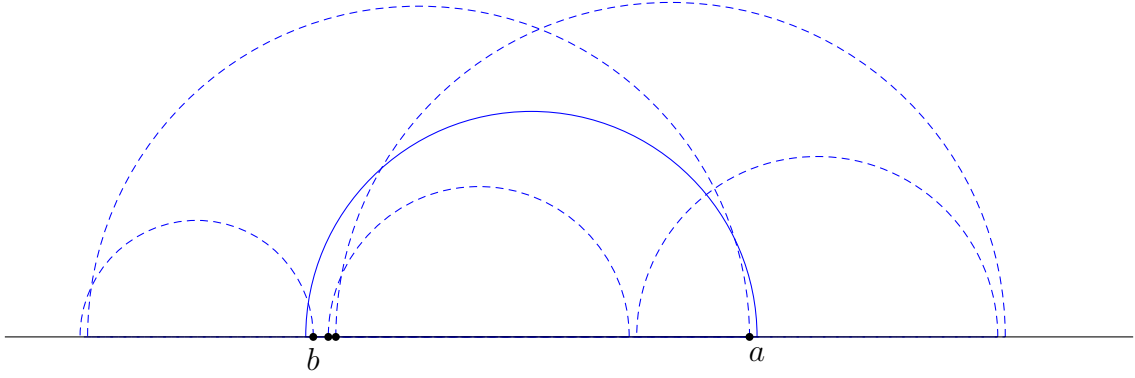
\begin{figure}
      \begin{tikzpicture}[scale=1]
        \hpline{4}{10}
        \hplinesty[densely dashed]{1.0}{4.1}
        \hplinesty[densely dashed]{4.3}{8.3}
        \hplinesty[densely dashed]{8.4}{13.2}
        \hplinesty[densely dashed]{1.1}{9.9}
        \hplinesty[densely dashed]{4.4}{13.3}
        \draw (0,0) -- (15,0);
        \coordinate (A) at (10,0);
        \fill (A)  node[below] {$a$};
        \coordinate (B) at (4.1,0);
        \fill (B)  node[below] {$b$};
        \coordinate (0) at (4.1,0);
        \fill (0) circle (1.5pt) node {};
        \coordinate (1) at (4.3,0);
        \fill (1) circle (1.5pt) node {};
        \coordinate (2) at (4.4,0);
        \fill (2) circle (1.5pt) node {};
        \coordinate (3) at (9.9,0);
        \fill (3) circle (1.5pt) node {};
      \end{tikzpicture}

  \caption{The black dots on $\partial \mathbb{D}$ count the sum (mod $2$) of the number of intersections with the solid chord $C$, that is, the $C$ summand on the left hand side of (\ref{crossing count}). By the notation of the proof of Proposition \ref{weighted cross counting}, the number of dots is $\order(s) + \overleftarrow{\order}(s')$.}
  \label{crossing count fig}
\end{figure}
Summing over all chords $C$ meeting $a$, we find that (\ref{s'}) sums to
   \[
     \sum_{t'\in \partial_a}(0 + 1 + \ldots + (m(t'_{\out})-1)) = \sum_{t'\in \partial_a}\frac{m(t'_{\out})}{2}
   \]
   (the identity holds since $m(t'_{\out})$ is even), and (\ref{s}) sums to
   \[
     \sum_{\substack{b,\\t\in \partial_b}}\sum_{s\in \partial_{a_b}(t_{\out},\alpha)}\overleftarrow{\order}(s)\pmod 2.\qedhere
   \]
   
\end{proof}
\printglossaries
\bibliography{Bibliography} 

@article {MR3230817,
    AUTHOR = {Lipshitz, Robert and Sarkar, Sucharit},
     TITLE = {A {K}hovanov stable homotopy type},
   JOURNAL = {J. Amer. Math. Soc.},
  FJOURNAL = {Journal of the American Mathematical Society},
    VOLUME = {27},
      YEAR = {2014},
    NUMBER = {4},
     PAGES = {983--1042},
      ISSN = {0894-0347,1088-6834},
   MRCLASS = {57M25 (55P42)},
  MRNUMBER = {3230817},
MRREVIEWER = {Nikolai\ N.\ Saveliev},
       DOI = {10.1090/S0894-0347-2014-00785-2},
       URL = {https://doi.org/10.1090/S0894-0347-2014-00785-2},
}

@article {MR4153651,
    AUTHOR = {Lawson, Tyler and Lipshitz, Robert and Sarkar, Sucharit},
     TITLE = {Khovanov homotopy type, {B}urnside category and products},
   JOURNAL = {Geom. Topol.},
  FJOURNAL = {Geometry \& Topology},
    VOLUME = {24},
      YEAR = {2020},
    NUMBER = {2},
     PAGES = {623--745},
      ISSN = {1465-3060,1364-0380},
   MRCLASS = {55P42 (57K18)},
  MRNUMBER = {4153651},
MRREVIEWER = {Gerd\ Laures},
       DOI = {10.2140/gt.2020.24.623},
       URL = {https://doi.org/10.2140/gt.2020.24.623},
}

@incollection {MR3611723,
    AUTHOR = {Lawson, Tyler and Lipshitz, Robert and Sarkar, Sucharit},
     TITLE = {The cube and the {B}urnside category},
 BOOKTITLE = {Categorification in geometry, topology, and physics},
    SERIES = {Contemp. Math.},
    VOLUME = {684},
     PAGES = {63--85},
 PUBLISHER = {Amer. Math. Soc., Providence, RI},
      YEAR = {2017},
      ISBN = {978-1-4704-2821-1},
   MRCLASS = {57M27 (18D05)},
  MRNUMBER = {3611723},
MRREVIEWER = {Stefan\ K.\ Friedl},
}

@article {MR2250492,
    AUTHOR = {Plamenevskaya, Olga},
     TITLE = {Transverse knots and {K}hovanov homology},
   JOURNAL = {Math. Res. Lett.},
  FJOURNAL = {Mathematical Research Letters},
    VOLUME = {13},
      YEAR = {2006},
    NUMBER = {4},
     PAGES = {571--586},
      ISSN = {1073-2780},
   MRCLASS = {57R17 (57M25 57M27)},
  MRNUMBER = {2250492},
MRREVIEWER = {Scott\ Morrison},
       DOI = {10.4310/MRL.2006.v13.n4.a7},
       URL = {https://doi.org/10.4310/MRL.2006.v13.n4.a7},
}

@article {MR2186113,
    AUTHOR = {Ng, Lenhard},
     TITLE = {A {L}egendrian {T}hurston-{B}ennequin bound from {K}hovanov
              homology},
   JOURNAL = {Algebr. Geom. Topol.},
  FJOURNAL = {Algebraic \& Geometric Topology},
    VOLUME = {5},
      YEAR = {2005},
     PAGES = {1637--1653},
      ISSN = {1472-2747,1472-2739},
   MRCLASS = {57M27 (53D12 57R17)},
  MRNUMBER = {2186113},
MRREVIEWER = {J\'er\^ome\ Petit},
       DOI = {10.2140/agt.2005.5.1637},
       URL = {https://doi.org/10.2140/agt.2005.5.1637},
}

@article {MR22071,
    AUTHOR = {Steenrod, N. E.},
     TITLE = {Products of cocycles and extensions of mappings},
   JOURNAL = {Ann. of Math. (2)},
  FJOURNAL = {Annals of Mathematics. Second Series},
    VOLUME = {48},
      YEAR = {1947},
     PAGES = {290--320},
      ISSN = {0003-486X},
   MRCLASS = {56.0X},
  MRNUMBER = {22071},
MRREVIEWER = {B.\ Eckmann},
       DOI = {10.2307/1969172},
       URL = {https://doi.org/10.2307/1969172},
}

@article {MR3252965,
    AUTHOR = {Lipshitz, Robert and Sarkar, Sucharit},
     TITLE = {A {S}teenrod square on {K}hovanov homology},
   JOURNAL = {J. Topol.},
  FJOURNAL = {Journal of Topology},
    VOLUME = {7},
      YEAR = {2014},
    NUMBER = {3},
     PAGES = {817--848},
      ISSN = {1753-8416,1753-8424},
   MRCLASS = {57M25 (55P42 55S10)},
  MRNUMBER = {3252965},
MRREVIEWER = {Nikolai\ N.\ Saveliev},
       DOI = {10.1112/jtopol/jtu005},
       URL = {https://doi.org/10.1112/jtopol/jtu005},
}

@article {MR4777698,
    AUTHOR = {Stoffregen, Matthew and Zhang, Melissa},
     TITLE = {Localization in {K}hovanov homology},
   JOURNAL = {Geom. Topol.},
  FJOURNAL = {Geometry \& Topology},
    VOLUME = {28},
      YEAR = {2024},
    NUMBER = {4},
     PAGES = {1501--1585},
      ISSN = {1465-3060,1364-0380},
   MRCLASS = {57K18 (55P91)},
  MRNUMBER = {4777698},
MRREVIEWER = {William\ Rushworth},
       DOI = {10.2140/gt.2024.28.1501},
       URL = {https://doi.org/10.2140/gt.2024.28.1501},
}

@article {MR3189434,
    AUTHOR = {Lipshitz, Robert and Sarkar, Sucharit},
     TITLE = {A refinement of {R}asmussen's {$S$}-invariant},
   JOURNAL = {Duke Math. J.},
  FJOURNAL = {Duke Mathematical Journal},
    VOLUME = {163},
      YEAR = {2014},
    NUMBER = {5},
     PAGES = {923--952},
      ISSN = {0012-7094,1547-7398},
   MRCLASS = {57M25 (55P42)},
  MRNUMBER = {3189434},
MRREVIEWER = {Laurence\ R.\ Taylor},
       DOI = {10.1215/00127094-2644466},
       URL = {https://doi.org/10.1215/00127094-2644466},
}

@unpublished{rajapakse2025,
      title={On {S}teenrod squares for even and odd {K}hovanov homology}, 
      author={Advika Rajapakse},
      year={2025},
      eprint={2509.03396},
      archivePrefix={arXiv},
      primaryClass={math.GT},
      url={https://arxiv.org/abs/2509.03396},
      note={(Preprint)},
}

@article {MR4165986,
    AUTHOR = {Lobb, Andrew and Orson, Patrick and Sch\"utz, Dirk},
     TITLE = {Khovanov homotopy calculations using flow category calculus},
   JOURNAL = {Exp. Math.},
  FJOURNAL = {Experimental Mathematics},
    VOLUME = {29},
      YEAR = {2020},
    NUMBER = {4},
     PAGES = {475--500},
      ISSN = {1058-6458,1944-950X},
   MRCLASS = {57K18},
  MRNUMBER = {4165986},
MRREVIEWER = {Nikolai\ N.\ Saveliev},
       DOI = {10.1080/10586458.2018.1482805},
       URL = {https://doi.org/10.1080/10586458.2018.1482805},
}

@article {MR4473678,
    AUTHOR = {Medina-Mardones, Anibal M.},
     TITLE = {New formulas for cup-{$i$} products and fast computation of
              {S}teenrod squares},
   JOURNAL = {Comput. Geom.},
  FJOURNAL = {Computational Geometry. Theory and Applications},
    VOLUME = {109},
      YEAR = {2023},
     PAGES = {Paper No. 101921, 16},
      ISSN = {0925-7721,1879-081X},
   MRCLASS = {55S10 (68U05)},
  MRNUMBER = {4473678},
MRREVIEWER = {Prasit\ Bhattacharya},
       DOI = {10.1016/j.comgeo.2022.101921},
       URL = {https://doi.org/10.1016/j.comgeo.2022.101921},
}

@article {MR1740682,
    AUTHOR = {Khovanov, Mikhail},
     TITLE = {A categorification of the {J}ones polynomial},
   JOURNAL = {Duke Math. J.},
  FJOURNAL = {Duke Mathematical Journal},
    VOLUME = {101},
      YEAR = {2000},
    NUMBER = {3},
     PAGES = {359--426},
      ISSN = {0012-7094,1547-7398},
   MRCLASS = {57M27 (57R56)},
  MRNUMBER = {1740682},
       DOI = {10.1215/S0012-7094-00-10131-7},
       URL = {https://doi.org/10.1215/S0012-7094-00-10131-7},
}

@article {MR1917056,
    AUTHOR = {Bar-Natan, Dror},
     TITLE = {On {K}hovanov's categorification of the {J}ones polynomial},
   JOURNAL = {Algebr. Geom. Topol.},
  FJOURNAL = {Algebraic \& Geometric Topology},
    VOLUME = {2},
      YEAR = {2002},
     PAGES = {337--370},
      ISSN = {1472-2747,1472-2739},
   MRCLASS = {57M27},
  MRNUMBER = {1917056},
MRREVIEWER = {Jacob\ Andrew\ Rasmussen},
       DOI = {10.2140/agt.2002.2.337},
       URL = {https://doi.org/10.2140/agt.2002.2.337},
}

@article {MR4076631,
    AUTHOR = {Piccirillo, Lisa},
     TITLE = {The {C}onway knot is not slice},
   JOURNAL = {Ann. of Math. (2)},
  FJOURNAL = {Annals of Mathematics. Second Series},
    VOLUME = {191},
      YEAR = {2020},
    NUMBER = {2},
     PAGES = {581--591},
      ISSN = {0003-486X,1939-8980},
   MRCLASS = {57K10 (57R65)},
  MRNUMBER = {4076631},
MRREVIEWER = {Laurence\ R.\ Taylor},
       DOI = {10.4007/annals.2020.191.2.5},
       URL = {https://doi.org/10.4007/annals.2020.191.2.5},
}

@incollection {MR4772951,
    AUTHOR = {Lipshitz, Robert and Sarkar, Sucharit},
     TITLE = {Khovanov homology of strongly invertible knots and their
              quotients},
 BOOKTITLE = {Frontiers in geometry and topology},
    SERIES = {Proc. Sympos. Pure Math.},
    VOLUME = {109},
     PAGES = {157--182},
 PUBLISHER = {Amer. Math. Soc., Providence, RI},
      YEAR = {[2024] \copyright 2024},
      ISBN = {[9781470470876]; [9781470477585]},
   MRCLASS = {57K18 (55N91 55P42 57K45)},
  MRNUMBER = {4772951},
MRREVIEWER = {William\ Rushworth},
}

@article {MR4889247,
    AUTHOR = {Baldwin, John A. and Hu, Ying and Sivek, Steven},
     TITLE = {Khovanov homology and the cinquefoil},
   JOURNAL = {J. Eur. Math. Soc. (JEMS)},
  FJOURNAL = {Journal of the European Mathematical Society (JEMS)},
    VOLUME = {27},
      YEAR = {2025},
    NUMBER = {6},
     PAGES = {2443--2465},
      ISSN = {1435-9855,1435-9863},
   MRCLASS = {57K18 (57K10 57K20 57K33)},
  MRNUMBER = {4889247},
       DOI = {10.4171/jems/1415},
       URL = {https://doi.org/10.4171/jems/1415},
}

@article {MoranHigherSquares,
    AUTHOR = {Cantero Mor\'an, Federico},
     TITLE = {Higher {S}teenrod squares for {K}hovanov homology},
   JOURNAL = {Adv. Math.},
  FJOURNAL = {Advances in Mathematics},
    VOLUME = {369},
      YEAR = {2020},
     PAGES = {107153, 79},
      ISSN = {0001-8708,1090-2082},
   MRCLASS = {55P42 (18N50 55S10 57K18)},
  MRNUMBER = {4094757},
MRREVIEWER = {Nikolai\ N.\ Saveliev},
       DOI = {10.1016/j.aim.2020.107153},
       URL = {https://doi.org/10.1016/j.aim.2020.107153},
}

@article {MR2729272,
    AUTHOR = {Rasmussen, Jacob},
     TITLE = {Khovanov homology and the slice genus},
   JOURNAL = {Invent. Math.},
  FJOURNAL = {Inventiones Mathematicae},
    VOLUME = {182},
      YEAR = {2010},
    NUMBER = {2},
     PAGES = {419--447},
      ISSN = {0020-9910,1432-1297},
   MRCLASS = {57M27},
  MRNUMBER = {2729272},
MRREVIEWER = {William\ D.\ Gillam},
       DOI = {10.1007/s00222-010-0275-6},
       URL = {https://doi.org/10.1007/s00222-010-0275-6},
}
\bibliographystyle{alpha}
\end{document}